\documentclass[11pt]{amsart}
\usepackage{amsthm, amssymb, amstext, amscd, amsfonts, amsxtra, latexsym, amsmath, comment, mathrsfs, esint, bm}
\usepackage{filecontents}

\usepackage{soul}
\usepackage{a4wide}
\usepackage[bottom=2.5cm, left=2.5cm, right=2.5cm]{geometry}
\usepackage[all]{xy}

\usepackage[style=numeric, backend=bibtex, doi=false, url=false, firstinits=true, eprint=false]{biblatex}
\bibliography{ripe}
\usepackage[pdfborder={0 0 0}, colorlinks=true, linkcolor=DarkBlue, citecolor=DarkGreen]{hyperref}
\usepackage{stmaryrd}

\usepackage{enumitem}
\usepackage[rgb,svgnames]{xcolor}
\usepackage[latin1]{inputenc}
\usepackage{enumitem}

\newtheorem{theorem}{Theorem}
\newtheorem{lemma}[theorem]{Lemma}
\newtheorem{corollary}[theorem]{Corollary}

\newtheorem{proposition}[theorem]{Proposition}

\theoremstyle{definition}
\newtheorem{example}{Example}

\newtheorem{definition}[theorem]{Definition}

\theoremstyle{remark}

\newtheorem*{remark}{Remark}
\newtheorem*{remarks}{Remarks}
\numberwithin{theorem}{section} \numberwithin{equation}{section}

\renewcommand{\Re}{\mathrm{Re}}
\renewcommand{\Im}{\mathrm{Im}}
\newcommand{\calF}{\mathcal{F}}

\newcommand{\reg}{\mathrm{reg}}

\DeclareMathOperator{\Ein}{Ein}
\DeclareMathOperator{\prj}{prj}

\newcommand{\calH}{\mathcal{H}}

\newcommand{\calL}{\mathcal{L}}
\newcommand{\calM}{\mathcal{M}}

\newcommand{\calQ}{\mathcal{Q}}

\newcommand{\calT}{\mathcal{T}}

\newcommand{\Mp}{\text {\rm Mp}}

\newcommand{\R}{\mathbb{R}}
\newcommand{\C}{\mathbb{C}}

\newcommand{\Q}{\mathbb{Q}}

\newcommand{\Z}{\mathbb{Z}}
\newcommand{\N}{\mathbb{N}}
\newcommand{\SL}{{\text {\rm SL}}}

\DeclareMathOperator{\Aut}{Aut}
\newcommand{\sgn}{\operatorname{sgn}}

\newcommand{\ES}{{\mathcal E}}

\newcommand{\tr}{{\text {\rm tr}}}

\newcommand{\ov}{\overline}
\DeclareMathOperator{\Log}{Log}

\DeclareMathOperator{\erfc}{erfc}

\DeclareMathOperator*{\CT}{CT}
\def\H{\mathbb{H}}
\newcommand{\abs}[1]{\left\vert#1\right\vert}
\newcommand{\bs}{\backslash}

\newcommand{\sig}{\operatorname{sig}}

\newcommand{\frake}{\mathfrak e}

\newcommand{\kzxz}[4]{\left(\begin{smallmatrix} #1 & #2 \\ #3 & #4\end{smallmatrix}\right) }
\newcommand{\kabcd}{\kzxz{a}{b}{c}{d}}

\newcommand{\bea}{\begin{eqnarray}}
\newcommand{\eea}{\end{eqnarray}}
\newcommand{\be}{\begin {equation}}
\newcommand{\ee}{\end{equation}}

\def\ca{{\mathfrak a}}

\newcommand{\la}{\langle}
\newcommand{\ra}{\rangle}

\newcommand\indlim{{\displaystyle\lim_{\longrightarrow}\,}}
\newcommand\hol{{\mathcal{O}}}

\begin{document}
\title{Regularized inner products and errors of modularity}\noindent
\author{Kathrin Bringmann}
\address{Mathematical Institute, University of Cologne, Weyertal 86-90, 50931 Cologne, Germany}
\email[K.~Bringmann]{kbringma@math.uni-koeln.de}

\author{Nikolaos Diamantis}
\address{School of Mathematical Science, University of Nottingham, Nottingham NG7 2RD, UK}
\email[N.~Diamantis]{nikolaos.diamantis@nottingham.ac.uk}

\author{Stephan Ehlen}
\address{McGill University, Department of Mathematics and Statistics, 805 Sherbrooke St. West, Montreal, Quebec, Canada H3A 0B9}
\email[S.~Ehlen]{stephan.ehlen@mcgill.ca}

\thanks{The research of the first author is supported by the Alfried Krupp Prize for Young University Teachers of the Krupp foundation and the research leading to these results receives funding from the European Research Council under the European Union's Seventh Framework Programme (FP/2007-2013) / ERC Grant agreement n. 335220 - AQSER}

\date{\today}
\begin{abstract}
  We develop a regularization for Petersson inner products of arbitrary weakly holomorphic modular forms, generalizing several known regularizations. As one application, we extend work of Duke, Imamoglu, and Toth on regularized inner products of weakly holomorphic modular forms of weights $0$ and $3/2$. These regularized inner products can be evaluated in terms of the coefficients of holomorphic parts of harmonic Maass forms of dual weights. Moreover, we study the errors of modularity of the holomorphic parts of such a harmonic Maass forms and show that they induce cocyles in the first parabolic cohomology group introduced by Bruggeman, Choie, and the second author. This provides explicit representatives of the cohomology classes constructed abstractly and in a very general setting in their work.
\end{abstract}

\maketitle

\section{Introduction and statement of results}
\label{sec:introduction}
Some of the most fundamental techniques towards arithmetic and
geometric applications of cusp forms are based on the Hilbert space
structure induced by the Petersson inner product. Recall that for
holomorphic cusp forms $f,g$ of weight $k \in \R$ for
$\SL_2(\Z)$, this is defined as
\begin{equation}\label{petint}
 \langle f,g\rangle := \int_{\SL_2(\Z)\bs \H} f(\tau) \overline{g(\tau)} v^k\,  \frac{dudv}{v^2},
\end{equation}
where throughout $\tau
 =u+iv$.
The integral converges absolutely if $fg$ is a cusp form.
There are ways to regularize the integral in many cases beyond cusp forms, as for example,
observed by Petersson \cite{Pe}.
Extensions and variants of his idea have been used by
Harvey and Moore \cite{HM}, Borcherds \cite{BoAutGra}, and Bruinier \cite{Br}
to regularize theta lifts of weakly holomorphic modular forms and harmonic Maass forms.

More recently, Duke, Imamo\={g}lu and T\'oth
\cite{dit-weight2} {used} regularized inner products to obtain interesting arithmetic information about
elements of the space $M_0^!$ of weakly holomorphic modular forms of weight $0$ for the full modular group.
To state their beautiful formula, for $m \in \N$, let $f_m$ be the unique modular function for $\SL_2(\Z)$ with
a Fourier expansion of the form ($ q:=e^{2 \pi i \tau}$)
\begin{equation}\label{fm}
 f_m (\tau) =q^{-m} +\sum_{n \ge 0} c_m (n) q^n
\end{equation}
and $c_m(0) = 24 \sum_{d\mid m} d$.
The forms $\{f_m\}_{m \in\N}$ together with the constant function $1$ form a basis of $M_0^!$.
Moreover, let
 \[
 K(m,n;c):= \sum_{d\pmod{c}^*} e^{\frac{2\pi i}{c} \left( md+n\overline{d}\right)},
 \]
 where the sum only runs over those $d\pmod{c}$
 that are coprime to $c$ and $\overline{d}$ denotes the multiplicative inverse of $d\pmod{c}$.
In Theorem 1 of \cite{dit-weight2} it was shown that $\la f_m, f_n \ra$ 
can be regularized for $m \neq n$ such that
the regularized value equals
\begin{equation}\label{fmfn}
    \langle f_m, f_n \rangle
     = -8\pi^2 \sqrt{mn} \sum_{c\geq 1} \frac{K(m,n;c)}{c} F\left(\frac{4\pi \sqrt{mn}}{c}\right).
  \end{equation}
 Here, $F(x):= \pi Y_1 (x) +\frac{2}{x} J_0 (x)$
 with $J_s$ the Bessel function of the first kind and $Y_s$ the Bessel
 function of the second kind.

In this paper we answer an open problem from \cite{dit-weight2}, namely to regularize
$\left\langle f_m ,f_n \right\rangle$ in the case if $m=n$ and to find a closed formula for this quantity. To achieve this goal,
we find a general regularization which works for all half-integral weight weakly holomorphic modular forms.
To explain this, recall that for every weakly holomorphic modular form $f \in M_k^!$, there exists
a harmonic Maass form $F$ of weight $2-k$ with $\xi_{2-k}(F) = f$, where
$\xi_{2-k}$ is the differential operator introduced by Bruinier and Funke \cite{BrF}
\begin{equation} \label{xi}
  \xi_{2-k}(f)(\tau) := 2iv^{2-k} \overline{\frac{\partial}{\partial\overline{\tau}} f\left(\tau\right)}.
\end{equation}
The Fourier expansion of a harmonic Maass form can be decomposed into a holomorphic and a non-holomorphic part.
We write $c_f(n)$ for the $n$-th Fourier coefficient of $f$ and $c_F^+(n)$ for the $n$-th Fourier coefficient of the holomorphic part of $F$.
The following theorem summarizes our first main result in the case of the full modular group.
It is a generalization of Proposition 3.5 in \cite{BrF} to the space of  weakly holomorphic modular forms.
Note that we work with vector-valued modular forms of integral and
half-integral weight in the main body of the paper (see Theorems \ref{thm:fgtindep} and \ref{InnProdPairing} for the general results).
\begin{theorem}\label{thm:petintro}
For $k \in 2\Z$, the Petersson inner product extends to a hermitian sesquilinear form
on $M_k^!$. If $f, g \in M_k^!$ and $G$ is a harmonic Maass form of weight $2-k$ such that $\xi_{2-k}(G) = g$, then we have
\[
 \langle f, g \rangle = \sum_{n \in \Z} c_f(n)c_G^+(-n).
\]
\end{theorem}

\begin{remark}
  The extension of the Petersson inner product is often degenerate and
  also not positive in general. Thus, strictly speaking, it is not an inner product.
  Nevertheless, as is common in the literature (see e.g. \cite{dit-weight2}),
 we refer to it as a {\it regularized inner product}.
It follows from Theorem \ref{thm:petintro} that the subspace of weakly holomorphic forms $f \in M_k^!$
with $\la f,g \ra = 0$ for all $g \in M_k^!$ is given by all $\xi_{2-k}(F)$, where $F$ is a harmonic Maass form of
weight $2-k$ with vanishing holomorphic part $F^+$.
\end{remark}

A first application of Theorem \ref{thm:petintro} is
\begin{theorem} \label{thm: fmfn}
  Equation \eqref{fmfn} also holds for $m=n$.
 \end{theorem}

A second application of our regularization concerns the case of weight $3/2$ modular forms discussed in \cite{DIT}.
For every $d\in \N$ there exists a unique form $g_d \in M_{3/2}^!$, the space of weakly holomorphic weight $3/2$ modular forms satisfying  Kohnen's plus space condition,
with Fourier expansion
 \[
 g_d (\tau) = q^{-d} +\sum_{\substack{n\geq 0 \\ n\equiv 0,3 \pmod{4}}}  B_d (n) q^n.
 \]
It was shown by Zagier \cite{zagier-traces} that the coefficients $B_d (n)$ are integers given
by (twisted) traces of singular moduli.
Duke, Imamo\={g}lu, and T\'oth proved in Theorem 2 of \cite{DIT} that for positive
fundamental discriminants $d\not =d'$,
the regularized inner product $\la g_d, g_{d'} \ra$ can be expressed in
terms of certain cycle integrals of the $j$-invariant.
As a consequence of our regularization, we are able to include the case
$d=d'$, which was not covered in \cite{DIT}. In particular, we obtain the following result.

\begin{theorem}\label{thm:g1intro}
  We have
  \begin{equation*}
    \la g_1, g_1 \ra = - \frac{3}{2\pi}\Re\left(\int_{i}^{i+1} J(\tau)\psi(\tau)\, d\tau\right),
  \end{equation*}
where $\psi(\tau) := \Gamma'(\tau)/\Gamma(\tau)$ denotes the Digamma function
and $J:= f_1 - 24$.
\end{theorem}

\begin{remarks} \
  \begin{enumerate}[leftmargin=*]
  \item The proof uses our extension of the regularization,
            the relation to Fourier coefficients of harmonic Maass forms as explained below, and
            Theorem 1.2 of \cite{BIF}.
  \item In Section \ref{connections}, we recall the definition of L-functions for weakly holomorphic modular forms.
            Theorem \ref{thm:g1intro} can also be stated as the identity $\la g_1, g_1 \ra = \frac{3}{4\pi} L_{f_1}^*(0)$,
            which is quite striking.
\item  A generalization of Theorem \ref{thm:g1intro} to all $d\in\N$ can be obtained
   by an extension of Theorem 1.2 in \cite{BIF} to twisted cycle integrals (which can be treated in a similar way as in \cite{AE}).
  \end{enumerate}
\end{remarks}

%Theorems \ref{thm: fmfn} and \ref{thm:g1intro} are intimately related to Fourier
%coefficients of holomorphic parts of harmonic Maass forms of dual weight.

We next further understand the role of the holomorphic part $F^+$ of
a harmonic Maass form $F$ of weight $2-k$ by providing a cohomological interpretation
for its error of modularity $$F_S:=F^+|_{2-k}(S-\mathrm{I}).$$
Here, $k\in -\frac12\N_0$, $S:=\left(\begin{smallmatrix} 0
&-1\\1&0\end{smallmatrix}\right)$, $\mathrm{I}:=\left(\begin{smallmatrix}
1 & 0\\0&1\end{smallmatrix}\right)$,
and $|_{2-k}$ is the usual action of $\SL_2(\Z)$ which we extend linearly to $\C[\SL_2(\Z)]$ (see also Section \ref{prelim}).
Note that for these considerations we concentrate on weakly holomorphic forms
of non-positive weight. Such forms have been studied less than their
positive weight counterparts.

The error of
modularity induces a cocycle in the cohomology group
$H^1_{\text{par}}(\SL_2(\Z), D)$ studied in \cite{BCD}. This
cohomology group characterizes the space $A_k(\SL_2(\Z))$ of all holomorphic
functions on $\H$, without any growth conditions, that are invariant under the $\mid_k$-action of $\SL_2(\Z)$.
This characterization (cf. Theorem E of \cite{BCD}) is encoded in an isomorphism
$$
\ES_k: A_k(\SL_2(\Z)) \longrightarrow H^1_{\text{par}}(\SL_2(\Z), D)
$$
and it fits into a program begun by Bruggeman, Lewis,
and Zagier \cite{BLZ1, BLZ2}
concerned with Eichler-Shimura-type theorems for very broad classes of
automorphic objects. Furthermore, we show that
the cohomology class induced by the error of modularity is a much
more central object than one might think at first; in fact,
it coincides with $\ES_k(\xi_{2-k}(F))$. In the special case
of integral weight, the result is
\begin{theorem}\label{coincIntr}
For $k\in -\N$ let $f\in M_k ^!$ and let $F$ be
a weight $2-k$ harmonic Maass form such that $\xi_{2-k}(F) = f$. The
cohomology class of $\ES_k(f)$ in $H^1_{\emph{par}}(\SL_2(\Z), D)$
equals the cohomology class of the cocycle induced by the error of
modularity $F_S$.
\end{theorem}

Theorem 1.3 also shows that $F_S$
provides  explicit
representatives of the cohomology classes constructed abstractly and in a
very general setting in \cite{BCD}. How to find examples of such explicit
representatives was an interesting open question once the general results
of \cite{BCD} were established.

Finally, we prove that the non-singular part
$\mathcal{G}_k(\tau)$
(defined by
\eqref{partdefined}) of the cocycle
induced by $F_S$
encodes further arithmetic information.
\begin{theorem} For $k\in-2\N$, the
$n$-th ``Taylor coefficient" of $\mathcal{G}_k(\tau)$, i.e.,
 $\mathcal{G}_k^{(n)}(0)/n!$
is (up to an explicit factor) equal to
\begin{equation}\label{LVintr}
L^*_f(n+1)+\textrm{explicit linear combination of $c_f(0), \dots,
c_f(m_0)$}
\end{equation}
where $L^*_f(s)$ is the L-function of $f\in M_k^!$
(defined in Section \ref{connections}) and $m_0\in\Z$ is minimal such that $c_f(m) \ne 0$.
\end{theorem}

\begin{remarks}\
  \begin{enumerate}[leftmargin=*]
  \item It is of interest to compare this result with those of \cite{BDR}.
In both cases we characterize values of L-functions as
``Taylor coefficients" of functions  with cohomological interpretations.
Also, in both papers  the  cocycles are errors of modularity of the
holomorphic (resp. harmonic) part of explicit harmonic (resp. sesquiharmonic) forms.
However, in \cite{BDR}, the L-values are associated to cusp forms, whereas here they are L-values of weakly holomorphic modular forms.
Another difference is that the relevant coefficient
module in \cite{BDR} is characterized by the action of the $\xi$-operator, whereas the coefficient
module here is defined by geometric means. Specifically, it is characterized by functions which are defined on a special type of domain.
In this sense, the constructions here can be seen as a
generalization of \cite{BDR}.
\item Finally, we note that Theorem \ref{Taylor} is reminiscent of the
main result of \cite{BO}.
In \cite{BO}, a harmonic Maass form is shown to be a ``generating
function" of (in their case) central values and derivatives
of quadratic twists of weight $2$ modular L-functions.
  \end{enumerate}
\end{remarks}
The paper is organized as follows. In the next section we recall the definitions of and basic facts about the
main objects which we study in the paper: vector-valued modular forms, harmonic Maass forms and also the special functions we require.
In Section \ref{sec:regprod}, we define the extension of the regularization of the Petersson inner product, and in Section
\ref{Relationto} we prove some of its fundamental properties. In particular, we prove Theorem
\ref{InnProdPairing} and deduce several applications, including
Theorems \ref{thm: fmfn} and \ref{thm:g1intro}. In Section \ref{coh}, we give a cohomological
interpretation of the error of modularity of the holomorphic part of a harmonic Maass form.
Finally, in Section \ref{connections}, we give an extended version of the definition of L-functions
of weakly holomorphic modular forms given in \cite{BFK, fricke-thesis} and prove Theorem \ref{LVintr}
about special values of such L-functions.

\section*{Acknowledgements}
The authors thank Roelof Bruggeman {and Jan Bruinier} for many
enlightening conversations. The authors also thank Claudia Alfes, Ben Kane, \'Arp\'ad T{\'o}th, and Mike Woodbury for providing comments on an earlier version of the paper.
The third author thanks the Max Planck Institute in Bonn, where he carried out most of the research related to this project, for providing an excellent and inspiring research environment during his visit (09--12/2015). Moreover we thank the referees for many helpful comments.

\section{Preliminaries}\label{prelim}
Unless stated otherwise, we have $k\in\frac{1}{2}\Z$ throughout.
Also, we use the principal branch (denoted by $\Log$) of the complex logarithm
with the convention that, for $x>0$,
$$
\Log(-x)=\log(x)+\pi i,
$$
where $\log: \R^+ \to \R$ is the natural logarithm.
This convention extends to all (complex and real) powers,
so that e.g.
$$\sqrt{-1} = (-1)^{\frac{1}{2}} = e^{\frac{1}{2}\Log(-1)} = e^{\frac{\pi i}{2}} = i.
$$

\subsection{Vector-valued modular forms and harmonic Maass forms}
We write $\widetilde{\Gamma} := \Mp_2(\Z)$ for the metaplectic extension of $\SL_2(\Z)$,
realized as the group of pairs $(M,\varphi(\tau))$, where $M=\kabcd\in\SL_2(\Z)$ and $\varphi$
is a holomorphic function on the complex upper half-plane $\H$ with $\varphi(\tau)^2=c\tau+d$ (see
e.~g.~\cite{BoAutGra,Br}).
It is well known that $\Mp_2(\Z)$ is generated by
$T:= \left(\left(\begin{smallmatrix} 1&1\\0&1 \end{smallmatrix}\right), 1\right)$
and $S:= \left(\left(\begin{smallmatrix}0&-1\\1&0\end{smallmatrix}\right), \sqrt{\tau}\right)$.
We have the relations $S^2=(ST)^3=Z$, where
$Z:=\left(\kzxz{-1}{0}{0}{-1}, i\right)$ is the standard generator of the center of $\Mp_2(\Z)$.
Let $\rho: \widetilde\Gamma \to \Aut(V)$ be a unitary, finite dimensional representation
factoring through a quotient by a finite index subgroup of $\widetilde\Gamma$ on a complex vector space $V$ with
hermitian inner product $(v_1, v_2)_V$ which we denote by $v_1\cdot \ov{v_2} := (v_1, v_2)_V$ 
for $v_1, v_2 \in V$. We write $\ov{V}$ for the complex conjugate of $V$ and let $\ov{\rho}$ 
be the complex conjugate of $\rho$, which acts on $\ov{V}$ via 
$\ov{\rho(A)}\, \ov{v} = \ov{\rho(A)v}$. Note that we obtain a $\C$-bilinear pairing
$V \times \ov{V} \to \C$ via $(v_1, \ov{v_2}) \mapsto (v_1, v_2)_V$, which we simply denote by
$v_1 \cdot v_2$.
We let $N_\rho$ be the smallest positive integer, such that $\rho(T)^{N_\rho}$ acts trivially on $V$.

For $(M,\varphi) \in \widetilde\Gamma$, we define the \emph{Petersson slash operator} on functions
$f: \H \rightarrow V$ by
\begin{equation*}
  \left( f \mid_{k,\rho} (M,\varphi)\right) (\tau) := \varphi(\tau)^{-2k} \rho((M, \varphi))^{-1} f(M \tau).
\end{equation*}
As usual, we extend the action linearly to $\C[\Mp_2(\Z)]$, e.g.
$f\mid_{k,\rho}(A+B) = f\mid_{k,\rho}A + f\mid_{k,\rho}B$.
\begin{definition}
  A twice continuously differentiable function $F: \H \to V$
  is called a {\it harmonic Maass form} of weight $k$ for $\rho$ if it
satisfies:
\begin{enumerate}[leftmargin=*]
\item
    $F \mid_{k,\rho} M = F$ for all $M \in \widetilde\Gamma$;
\item
    there exists a $C>0$ such that $F(\tau)=O(e^{C v})$ as $v\to \infty$
    (uniformly in $u$);
\item
    $\Delta_k (F) = 0$, with $\Delta_k$ the hyperbolic Laplace operator
    \begin{align*}
      \Delta_k := -v^2\left( \frac{\partial^2}{\partial u^2}+
      \frac{\partial^2}{\partial v^2}\right) + ikv\left(\frac{\partial}{\partial u}+i \frac{\partial}{\partial v}\right).
    \end{align*}
  \end{enumerate}
\end{definition}

\begin{remark}
  \item Here and throughout, we use the notation $a(x) = O(\phi(x))$ for functions $a: \R \to V$
and $\phi: \R \to \R$, if $a(x) \cdot \ov{w} = O(\phi(x))$ for every fixed $w \in V$.
\end{remark}

We denote the space of harmonic Maass forms of weight $k$ for $\rho$ by
$\calH_{k,\rho}$. An element $F \in \calH_{k,\rho}$ has a Fourier expansion,
\begin{equation} \label{deff}
  F(\tau) = \sum_{n \in \Q} c_{F}(n, v) q^n,
\end{equation}
with $c_F(n,v) \in V$. The right hand side of \eqref{deff} decomposes  into a holomorphic and a non-holomorphic part. 
To accomplish this, for $x\in\mathbb{R}$, set
\begin{equation*}
\label{defineW}
W_k(x):= (-2x)^{1-k} \Re(E_k(-2x))
\end{equation*}
with $E_k$ the generalized
exponential integral defined in Section \ref{sec:sf}.
Note that our definition of $W_k$ slightly differs
from the one made in \cite{BrF},
but this is only relevant for $n>0$, which is not
the main focus of their paper. In Section \ref{sec:sf} we determine the
exact difference and also clarify the relation to the choice made in
\cite{dit-weight2} in the specific case of weight $2$.
Now the decomposition $F = F^+ + F^-$ for $k \neq 1$ is given by
\begin{align}
  F^+(\tau)&= \sum_{\substack{n\in \Q\\ n\gg-\infty}} c_{F}^+(n) q^n,
\label{deff+}\\
  F^-(\tau) &= c_{F}^-(0)v^{1-k} + \sum_{\substack{n\in \Q \setminus\{0\}
\\ n \ll \infty}} c_{F}^-(n) W_k(2\pi n v) q^n \label{deff-}. \end{align}
For $k=1$, $F^-$ has to be replaced by
\begin{equation}
  F^-(\tau) = c_{F}^-(0) \log(v)
    + \sum_{\substack{n\in \Q \backslash \lbrace 0 \rbrace \\ n \ll
\infty}} c_{F}^-(n) W_k(2\pi nv) q^n \label{f-k1}.
\end{equation}
The function $F^+$ is called the \emph{holomorphic part} and $F^-$ the \emph{non-holomorphic part} of the harmonic Maass form $F$.
We call the Fourier polynomial
\[
  P_F(q) = P_F^+(q) = \sum_{n<0} c_F^+(n)q^n
\]
the (holomorphic) \emph{principal part} of $F$.

If $F^-$ is identically zero, then $F$ is \emph{weakly holomorphic}.
The subspace of weakly holomorphic modular forms is denoted by
$M^!_{k,\rho}$ and we write $M_{k,\rho}$  and $S_{k,\rho}$  for the spaces of holomorphic modular forms and cusp forms,
respectively.
If $F \in M_{k,\rho}^!$, we simply write $c_F(n)$ for $c_F^+(n)$.

Recall the antilinear differential operator
$\xi_k: \mathcal{H}_{k,\rho} \to M^{!}_{2-k,\overline{\rho}}$, defined in (\ref{xi}).
Here $M^{!}_{2-k,\overline\rho}$ is the space of weakly holomorphic
modular forms of weight $2-k$ with respect to the complex conjugate
$\overline\rho$ of $\rho$.
The kernel of $\xi_k$ equals $M^!_{k,\rho}$ and by
Corollary~3.8 of \cite{BrF}, the sequence
\begin{gather}
  \label{ex-sequ}
  \xymatrix@1{ 0 \ar[r] & M_{k,\rho}^! \ar[r] & \calH_{k,\rho} \ar[r]^-{\xi_{k}} & M^{!}_{2-k,\overline\rho} \ar[r] & 0}
\end{gather}
is exact (note that in \cite{BrF}, $\rho$ is the Weil representation but the proof easily generalizes).

\begin{lemma} {(Lemma 3.1 of \cite{BrF})}
  For $k \neq 1$, the Fourier expansion of $\xi_k (F) \in M^{!}_{2-k,\overline\rho}$ for any
  $F \in \calH_{k,\rho}$ is given by\footnote{Note that there is a minor typo in \cite{BrF}.}
  \begin{equation}
\label{eq:xife}
      \overline{c_{F}^-(0)} (1-k)
      - \sum_{n\in\Q \setminus\{0\}} \sgn(n)^{k-1}\overline{c_{F}^-(-n)}(4\pi |n|)^{1-k} q^n.
  \end{equation}
For $k=1$, the first term is replaced by $-\ov{c_F^-(0)}$.
\end{lemma}
From \eqref{eq:xife} we deduce two identities that we frequently use. Namely, for
$f \in M_{k}^!$ and $F \in \calH_{2-k}$ with $\xi_{2-k}(F) = f$,  we have, for $n \neq 0$,
\begin{equation}
\label{usefuleq}
c_F^{-}(n)(-4\pi n)^{k-1} = -\ov{c_f(-n)} \quad \text{ and } \quad
c_F^{-}(0)(k-1) = \overline{c_f(0)}.
\end{equation}

The following well known growth estimates for the Fourier coefficients
of harmonic Maass forms can be proven as in Lemma 3.4. of \cite{BrF}.
\begin{lemma}
\label{lem:coeffgr}
  If $F \in \calH_{k,\rho}$, then there exists a constant $C > 0$, such that
  \[
    c_{F}^{\pm}(n) = O\left(e^{C\sqrt{\abs{n}}}\right)\,\qquad \textrm{as} \,\, n \to \pm\infty.
  \]
  Moreover, if $\xi_k(F)$ is cuspidal, then we have the stronger estimate
  \[
  c_{F}^{-}(n) = O\left(\abs{n}^{\frac{k}{2}}\right)\, \textrm{as} \,\, n \to -\infty.
  \]
\end{lemma}

We next recall an important family of representations
given by the Weil representation associated with a finite quadratic module.
Let $(A,Q)$ be a finite quadratic module (also called a finite quadratic form or discriminant form), that is,
a pair consisting of a finite abelian group $A$ together with a $\Q/\Z$-valued non-degenerate quadratic form $Q$ on $A$.
We denote the bilinear form corresponding to $Q$ by $(x,y) := Q(x+y) - Q(x) - Q(y)$.
Recall that $Q$ is called {\it degenerate} if there exists $x \in A\setminus\{0\}$, with $(x,y) = 0$ for all $y \in A$.
Otherwise, $Q$ is {\it non-degenerate}. If the quadratic form is clear from the context, then we simply write $A$
for the pair $(A,Q)$.
If $L$ is an even lattice, then the quadratic form $Q$ on $L$
induces a $\Q/\Z$-valued quadratic form on the discriminant group $L'/L$ of $L$.
Here $L'$ is the \emph{dual lattice} of $L$, given by
\[
  L' := \big\{ x\in L \otimes_\Z\Q\ \mid\ (x,\lambda) \in \Z \text{ for all } \lambda \in L \big\}.
\]
The pair $(L'/L,Q)$ defines a finite quadratic module, the \emph{discriminant module} of $L$.
According to Theorem 1.3.2 of \cite{Ni},
any finite quadratic module can be obtained as the discriminant module of an even lattice.
If $(b^+, b^-)$ denotes the real signature of $L$,
then the difference $b^+ - b^-$ is determined by its discriminant module $A:=L'/L$ modulo $8$ by Milgram's formula (Appendix 4 of \cite{MiHu})
\[
  \frac{1}{\sqrt{\abs{A}}} \sum_{a \in A} e(Q(a)) = e\left(\frac{b^{+} - b^{-}}{8}\right),
\]
where we set $e(x) := e^{2\pi i x}$.
We call $\sig(A) := b^{+} - b^{-}  \in \Z/8\Z$ the \emph{signature} of $A$.
We also let $N$ be the {\it level} of $A$ defined by
\[
  N := \min\{ n\in \N \mid\; \text{$nQ(x)\in \Z$ for all $x\in A$}\}.
\]

The \emph{Weil representation} $\rho_A$ associated with $A$ is a unitary representation of
$\Mp_2(\Z)$ on the group algebra $\C[A]$. If we denote the standard basis of
$\C[A]$ by $(\frake_a)_{a \in A}$, then $\rho_A$ can
be defined by the action of the generators $S,T \in \Mp_2(\Z)$ as
follows (see also \cite{BoAutGra,Br, skoruppa-jacobi-weil}):
\begin{align}\label{actionT}
  \rho_A(T)\frake_a &:= e(Q(\gamma)) \frake_a,\\
  \label{actionS}
  \rho_A(S)\frake_a &:= \frac{ e\left(\frac{-\sig(A)}{8}\right)}{\sqrt{|A|}} \sum_{b \in A} e(- (a,b)) \frake_b.
\end{align}
We write $a\cdot \ov{w}$ to denote the standard hermitian inner product on $\C[A]$ given by $\frake_a \cdot \ov{\frake_b} = \delta_{a, b}$ for $a, b \in A$.

We next give some examples that relate vector-valued and scalar-valued harmonic Maass forms in important cases.
\begin{example}
  If $(A,Q)$ is the trivial module $A = \{0\}$, then the representation $\rho_A$ is the trivial representation and
  modular forms (and harmonic Maass forms) of weight $k$ for $\rho_A$ coincide with scalar-valued modular forms (resp. harmonic Maass forms) for $\SL_2(\Z)$. We write $S_k \subset M_k \subset M_k^! \subset \calH_k$ to denote the corresponding spaces of scalar-valued forms.
 \end{example}
Under the assumptions made, the components of a vector-valued modular form are modular forms for a finite index subgroup of $\Mp_2(\Z)$.
Thus, we can obtain such scalar-valued modular forms from a vector-valued form in various ways.
\begin{example}
  For any $\bm{F} \in \calH_{k,\rho_A}$, the function
  \begin{equation}
    \label{eq:Ff}
    F(\tau) := \sum_{a \in A} \bm{F}_a(N\tau)
  \end{equation}
 is a scalar-valued harmonic Maass form of level $N$, weight $k$, and certain character $\chi_A$.
 The map $\bm{F} \mapsto F$ is often neither injective nor surjective.
 However, it respects all analytic conditions, i.e., it preserves subspaces of cusp forms,
 modular forms, and weakly holomorphic modular forms.
\end{example}
We next give a concrete example of this correspondence.
\begin{example}
\label{ex:G0NWrep}
  Let $M\in \N$ and $A := \Z/2M\Z$ with quadratic form $Q(x) := \frac{x^2}{4M}$.
  Then, for $\bm{F} \in \calH_{k,\rho_A}$, the function $F$ defined in \eqref{eq:Ff} is a harmonic Maass form for $\Gamma_0(4M)$ in Kohnen's plus-space,
  i.e., $c_{F}(n) = 0$ unless $n$ is a square modulo $4M$. For $M=1$ or $M=p$ prime and $k = 2k' +1/2$ with $k' \in \Z$, this map gives an isomorphism
  between the space $\calH_{k,\rho_A}$ and the space of scalar-valued
  harmonic Maass forms satisfying the plus-space condition.
  For $k = 2k' - 1/2$, the same statement holds for $\ov{\rho}$.
 In this case $c_F(n)=0$ unless $-n$ is a square modulo $4p$
  (and the space $M_{k,\ov{\rho}}$ is isomorphic to the space of Jacobi forms of weight $2k'$ and index $M$).
  See e.g. Example 2.4 of \cite{BoAutGra} and \S 5 of \cite{EZ}.
\end{example}

\subsection{Some special functions}
\label{sec:sf}
An important role in the construction of the regularized inner product
is played by generalized exponential integrals and related
functions. In this section, we recall their definitions and carefully consider the
choices of their branches.

Recall the definition of the incomplete Gamma function
\[
  \Gamma(r,z) := \int_{z}^\infty e^{-t}t^{r}\, \frac{dt}{t},
\]
initially given for $\Re(r) > 0$ and $z \in \C$. For $r \in \C$ and $z\in\mathbb{C}$ with $\Re(z)>0$,
we define the generalized exponential integral
$E_r$ (see \cite{NIST}, 8.19.3) by
\[
  E_r(z) := \int_1^{\infty} e^{-zt}t^{-r}\, dt.
\]
This function is related to the incomplete Gamma function via (8.19.1) of \cite{NIST}:
\begin{equation}\label{Gamma}
  \Gamma(r, z) = z^r E_{1-r}(z),
\end{equation}
which can be used to continue $E_r(z)$ analytically.
In particular, for $r\in \frac{1}{2}\Z$ or $r\in \N$, this can be done using the
relation (see \cite{NIST}, 8.19.12)
\begin{equation} \label{recurs}
  rE_{r+1}(z) + zE_r(z) = e^{-z}.
\end{equation}
Indeed, for $m\in \Z$ and $c\in \mathbb{R}^+$, we obtain with $\nu := \sgn(m)$ and $\mu := (1+\nu)/2$
\begin{equation} \label{eq:Enexp}
  E_{m+c}(z) = \frac{e^{-z}z^{\mu-1}}{\Gamma\left( c + m \right)}
\sum_{\ell=0}^{|m|-1}(-z)^{\nu \ell}\Gamma\left(c + m - \nu \ell - \mu\right)
+ \frac{\Gamma\left(c\right)(-z)^m}{\Gamma\left(c+m\right)} E_{c}(z).
\end{equation}
Here, the right-hand side should be interpreted as its analytic continuation, so that for negative $m+c$
the poles of the Gamma function cancel or give zeroes.
Thus, for $c=1$ or $c=1/2$,
continuing $E_{m+c}$ is equivalent to analytically continuing $E_1$ and $E_{1/2}$.
We may do this for $E_1$ by the identity (see \cite{NIST}, 6.2.4):
\begin{equation}
  \label{eq:E1exp}
  E_1(z) = \Ein(z) -  \Log(z) - \gamma,
\end{equation}
where $\gamma$ is the Euler-Mascheroni constant and
$\Ein$ is the entire function given by
 \[
  \Ein(z)  :=\int_0^{z} \left(1-e^{-t}\right)\frac{dt}{t}
                =\sum_{n=1}^{\infty}\frac{(-1)^{n+1}}{n!\, n} z^n.
\]
To continue $E_{1/2}$, we use identity 7.11.3 of \cite{NIST}
\begin{equation}\label{eq:erfc}
  E_{\frac12}(z) :=\frac{\sqrt{\pi}}{\sqrt{z}}\erfc \left (\sqrt z \right )
                         := \frac{2}{\sqrt{z}} \int_{\sqrt z}^{\infty} e^{-t^2}\, dt
                         = \frac{\sqrt{\pi}}{\sqrt{z}} - \frac{2}{\sqrt{z}} \int_{0}^{\sqrt{z}}e^{-t^2}\, dt.
\end{equation}
More generally, by selecting another branch of the
logarithm in \eqref{eq:E1exp} (resp. of $\sqrt{z}$ in \eqref{eq:erfc}), we may obtain a different analytic continuation of $E_r$
for $r \in \N$ (resp. $r \in \frac{1}{2}\Z$).
If the branch cut is given by the ray $\{xe^{i\varphi}|x \in \R^+_{0}\},$ then we
denote the corresponding continuation by $E_{r, \varphi}$.
We throughout use the abbreviation $E_{r}$ if $\varphi \in\pi \Z\setminus 2\pi \Z$, in which case the corresponding branch is the principal branch. We also require (see \cite{NIST}, 8.19.6), that for $\sigma :=\Re(s) > 1$
\begin{equation}
  \label{eq:Es0}
  E_{s}(0) = \frac{1}{s-1}.
\end{equation}

With the above notation, we can
relate the functions $W_k$ in \eqref{deff-} to the
special functions used in Section 3 of
\cite{BrF}.
\begin{lemma}\label{lem:relationsofnonhol}
For $x>0$, we have
  \begin{equation}\label{relationsofnonhol}
  W_k(x)  = (-1)^{1-k}\left((2x)^{1-k}E_k(-2x) + \frac{\pi i}{\Gamma(k)}\right)
               = \Gamma(1-k, -2x)+\frac{(-1)^{1-k}\pi i}{\Gamma(k)}.
\end{equation}
\end{lemma}
\begin{remark}
  Bruinier and Funke \cite{BrF} used the function $(-2x)^{1-k}E_k(-2x)$ instead of $W_k$.
  Lemma~\ref{lem:relationsofnonhol} shows that these two choices differ by an additive constant for $x>0$.
  For $x<0$, they agree since $E_k(-2x)$ is real in that case.
\end{remark}
\begin{proof}
With the choices made above, \eqref{eq:E1exp} implies that $\Im(E_1(-x)) = -\pi$.
Similarly, by \eqref{eq:erfc}, we obtain
\begin{equation}\label{ImE1/2}
\Im\left(E_{\frac12}(-x)\right) = -\sqrt{\frac{\pi}{x}}.
\end{equation}
Using \eqref{eq:Enexp}, this easily generalizes for $m \in \Z$  and $c\in\{1/2,1\}$ to
\begin{equation}
  \label{eq:ImEkint1}
  \Im(E_{m+c}(-x)) = -\frac{\pi x^{m+c-1}}{\Gamma(c+m)}.
\end{equation}
Finally, \eqref{eq:ImEkint1} implies that $\Im(E_{-m}(-x))=0$ for all $m \in \N_0$.
Substituting these identities into the formula of $W_k$, we deduce the lemma.
\end{proof}
\begin{remark}
  We also note that
\begin{equation}
\label{eq:ReE12}
   \Re\left(E_{\frac12}(-x)\right) = -2 \int_0^1 e^{x t^2}dt.
\end{equation}
\end{remark}

It is convenient to also use the $W$-notation for some complex arguments setting, for $z \in \C \setminus \R_0^{-}$,
\begin{equation}\label{Wcomplex}
  W_k(z) :=  \Gamma(1-k, -2z)+\frac{(-1)^{1-k}\pi i}{\Gamma(k)}.
\end{equation}

Furthermore, in order to relate our results to \cite{dit-weight2},
we compare $W_k(v)$ to the function $e^{2\pi n v}\calM_n(v)$ used in \cite{dit-weight2} for $k=2$.
Recall the definition
\begin{equation}
\mathfrak{M}(v):=\frac{d}{d s}\left[M_{1, s-\frac12}(v) - W_{1, s-\frac12}(v)\right]_{s=1}, \label{frakM}
\end{equation}
where $M_{a, b}$ and $W_{a, b}$ are the standard Whittaker functions given, for
Re$(b \pm a+1/2)>0$ and $v>0$, by

\begin{align}\label{MW}
M_{a, b}(v)&=v^{b+\frac{1}{2}}e^{\frac{v}{2}}\frac{\Gamma(1+2b)}{\Gamma\left(b+a+\frac{1}{2}\right)
\Gamma(b-a+\frac{1}{2})} \int_0^1 t^{b+a-\frac{1}{2}}(1-t)^{b-a-\frac{1}{2}}e^{-v t}dt, \\
W_{a, b}(v)&=v^{b+\frac{1}{2}}e^{\frac{v}{2}}\frac{1}{\Gamma\left(b+a+\frac{1}{2}\right)} \int_1^{\infty} t^{b+a-\frac{1}{2}}(t-1)^{b-a-\frac{1}{2}}e^{-v t}dt.\end{align}
Then, for $n \in \N,$ set
\begin{equation}
\mathcal{M}_n(v) = \frac1{4\pi v}\mathfrak{M}(4\pi nv)-(1-\gamma) n
e^{-2\pi nv}. \label{CalM}
\end{equation}
\begin{proposition} \label{prop:MnE2}
For $n\in\N$,  we have
\begin{equation*}
e^{2\pi n v}\mathcal{M}_n(v) = -\frac{1}{4 \pi v}E_2(-4 \pi n v)-\pi i n
= n W_2(2\pi n v).
\end{equation*}
\end{proposition}
\begin{proof}
For the first term of \eqref{CalM}, we use the definition \eqref{frakM} of $\mathfrak{M}(v)$.
Following the proof of Proposition 1 of \cite{dit-weight2}, we employ \eqref{MW} to write, for $Z\in\{M,W\}$,

$$Z_{1, s-\frac12}(v)=g_Z(s)h_Z(s),\text{ where}$$

\begin{align*}
g_M(s) &:= v^s e^{\frac{v}{2}}\frac{\Gamma(2s)}{\Gamma(s+1)\Gamma(s)}, \quad
g_W(s):=\frac{v^s e^{\frac{v}{2}}}{\Gamma(s)},\\
h_M(s) &:=(s-1)\int_0^1 t^s(1-t)^{s-2} e^{-vt}dt, \quad h_W(s):=(s-1)\int_1^\infty t^s(t-1)^{s-2} e^{-vt}dt.
\end{align*}
Then we have

\begin{equation}\label{diffM}
\frac{d}{ds}\left[Z_{1, s-\frac12}(v)\right]_{s =1}=g_Z^{\prime}(1)h_Z(1)+g_Z(1)h_Z^{\prime}(1).
\end{equation}

We evaluate
\[
g_M(1)         =  g_W(1)=v e^{\frac{v}{2}},\quad
g_M'(1)        =  v e^{\frac{v}{2}}\left(1+\log (v)\right), \quad
g_W'(1)       =  v e^{\frac{v}{2}}\left(\gamma+\log (v)\right).
\]
To compute $h_M$ and $h_M'$, we use integration by parts to obtain for $\sigma\gg 0$
$$
h_M(s)=\int_0^1(1-t)^{s-1}\frac{d}{dt}\left(t^s e^{-vt}\right)dt.
$$
From this we obtain, via analytic continuation, $h_M(1)= e^{-v}$.
Similarly, we determine that
\[
h_W(s)=-\int_1^\infty (t-1)^{s-1} \frac{d}{dt}\left(t^se^{-vt}\right)dt,
\]
and therefore $h_W(1)=e^{-v}$.
Thus the contribution to \eqref{frakM} from these terms is
$v e^{-\frac{v}{2}}\left(1- \gamma \right),$
which cancels the second term from \eqref{CalM} after making the required change of variables.

We next turn to the terms involving the derivatives $h_M'$ and $h_W'$.
We have
\[
h_M'(s)=\int_0^1 \log(1-t) (1-t)^{s-1} \frac{d}{dt}\left(t^s e^{-vt}\right)dt
		+\int_0^1  (1-t)^{s-1}\frac{d}{dt}\left(\log (t) t^s e^{-vt}\right)dt.
\]
This yields
\[
h_M'(1) = \int_0^1 \log(1-t)\frac{d}{dt} \left(te^{-vt}\right)dt.
\]
Similarly,
\[
h_W'(s) = -\int_1^\infty \log(t-1) (t-1)^{s-1} \frac{d}{dt}\left(t^s e^{-vt}\right)dt
		-\int_1^\infty  (t-1)^{s-1}\frac{d}{dt}\left(\log (t) t^s e^{-vt}\right)dt,
\]
which implies that
\[
h_W'(1) = - \int_1^\infty \log(t-1)\frac{d}{dt} \left(te^{-vt}\right)dt.
\]
This shows that the contribution from the derivatives of $h_M'$ and $h_W'$ to \eqref{frakM} is given by
\[
v e^{\frac{v}{2}}\left(\int_0^1\log (1-t) \frac{d}{dt} \left(te^{-vt}\right)dt+\int_1^\infty \log(t-1)\frac{d}{dt}\left(t e^{-vt}\right)dt\right).
\]
Combining the above yields
\begin{equation} \label{I} \mathcal{M}_n(v)=
ne^{2\pi nv} \int_0^1 \log\left(1-t\right)\frac{d}{dt}\left(t e^{-4\pi nvt}\right)dt + n
e^{2\pi n v}\int_1^\infty \log(t-1)\frac{d}{dt}\left(t e^{-4\pi n v
t}\right)dt.
\end{equation}

We rewrite each integral by viewing it as limit of integrals we may identify. First,
\begin{equation}
\int_0^1 \log\left(1-t\right)\frac{d}{dt}\left(t e^{-4\pi nvt}\right)dt=\lim_{a \to 1^-}
\int_0^a \log\left(1-t\right)\frac{d}{dt}\left(t e^{-4\pi nvt}\right)dt.
\end{equation}
Integration by parts yields for the integral on the right-hand side
\begin{equation}\label{2.27}
\log(1-a)e^{-4\pi n v}(a e^{-4\pi n v(a-1)}-1)+\frac{e^{-4\pi nva}-1}{4\pi n v}+
e^{-4 \pi n v}\int_1^{1-a} \frac{1-e^{4\pi nvt}}{t}dt.
\end{equation}
Upon taking the limit as $a \to 1^-$, this gives
\begin{equation}\label{I1}
\int_0^1 \log\left(1-t\right)\frac{d}{dt}\left(t e^{-4\pi nvt}\right)dt=\frac{e^{-4\pi nv}-1}{4\pi n v}-e^{-4 \pi n v} \text{Ein}(-4 \pi n v).
\end{equation}
\indent
In the same way,
\begin{equation}\label{infinit}
\int_a^{\infty} \log\left(t-1\right)\frac{d}{dt}\left(t e^{-4\pi nvt}\right)dt=
-\log(a-1)ae^{-4\pi n va}-\frac{e^{-4\pi nva}}{4\pi n v}-
e^{-4 \pi n v}\int_{a-1}^{\infty} \frac{e^{-4\pi nvt}}{t}dt.
\end{equation}
The last integral may be decomposed as
$$
-\int_{a-1}^1 \frac{1-e^{-4\pi nvt}}{t}dt-\log (a-1)+E_1(4 \pi n v).
$$
Thus \eqref{infinit} equals
\[
-\log(a-1)e^{-4\pi n v}(ae^{-4\pi n v(a-1)}-1)-\frac{e^{-4\pi nva}}{4\pi n v}+
e^{-4 \pi n v}\int_{a-1}^1 \frac{1-e^{-4\pi nvt}}{t}dt\\
-e^{-4 \pi n v}E_1(4 \pi n v),
\]
which, upon taking $a \to 1^+$, gives by (\ref{eq:E1exp})
\begin{equation}\label{2.31}
\int_1^{\infty} \log\left(t-1\right)\frac{d}{dt}\left(t e^{-4\pi nvt}\right)dt=
-\frac{e^{-4\pi nv}}{4\pi n v}+e^{-4 \pi n v} \Big ( \log (4\pi nv) +\gamma \Big ).
\end{equation}
Thus, adding \eqref{I1} and \eqref{2.31}, we obtain by (\ref{eq:E1exp}) the contribution to \eqref{I} as claimed.
\end{proof}

We require another special function, defined in equation (2.4) of \cite{dit-weight2} for $v>0$ and $n\in-\N$:
\begin{equation}\label{calW}
\mathcal{W}_n(v):=(-4 \pi n v)^{-1}W_{-1, \frac12}(-4 \pi n v)=
e^{-2 \pi n v} W_2(2 \pi n v),
\end{equation}
where the second equality follows by 13.18.5 of \cite{NIST}
and Lemma \ref{lem:relationsofnonhol}.

Finally, a straightforward calculation relates $E_{1/2}$ to the $\beta$-functions. 
%{\bf KB: do we need the following?}
\begin{lemma}\label{lem:beta12}
  We have, for $x>0$, that
\begin{align*}
  \beta_{\frac12}(x) &:= \int_{1}^\infty e^{-xt}t^{\frac12}\, dt = E_{\frac12}(x),\\
  \beta_{\frac12}^c(-2x) &:= \int_{0}^1 e^{2xt}t^{-\frac12}\, dt = -(-2x)^{\frac{1}{2}}\, W_{\frac{1}{2}}(x).
\end{align*}
\end{lemma}

\section{Regularized Petersson inner products}
\label{sec:regprod}
There are various ways to define regularized inner products.
For instance, if $f,g \in M_{k}^!$ satisfy $c_f(n)  c_g(n) = 0$ for $n  \leq 0$, then the inner product used in
 \cite{dit-weight2} is given by
\[ \langle f, g \rangle := \lim_{t \to \infty}\int_{\calF_t} f(\tau) \ov{g(\tau)}v^k\,\frac{dudv}{v^2},
\]
where $\calF_t := \{\tau\in\calF\mid v\leq t\}$ is the fundamental
domain truncated at height $t$ with $\calF$ the standard fundamental
domain for $\SL_2(\Z)$ on $\H$. However, this definition does not include
all cases, such as $\langle f, f \rangle$.\\
\noindent
In this paper, we define a generalization of the regularization that was used in [14]. Our generalization covers all weakly holomorphic forms. It builds upon the previously known methods but also deals with the cases where these fail.

In the following, we let $k \in \frac{1}{2}\Z$, $f, g \in M^!_{k,\rho}$, and $s, w \in \C$.
We set $\zeta:=\Re(w)$ and $\sigma:= \Re(s)$ and define
  \begin{equation}
\label{eq:Ifg} I(f, g; w, s) :=  \int_{\calF} f(\tau) \cdot \ov{g(\tau)}v^{k-s} e^{-wv}\, \frac{dudv}{v^2},
  \end{equation}
which certainly converges absolutely if $\zeta \gg 0$.
The aim of this section is to show that, for every $\varphi \in (\pi/2, 3\pi/2) \setminus \{\pi\}$,
the integral $I(f,g; w,s)$ has an analytic continuation $I_\varphi(f,g; w,s)$ to $U_\varphi \times \C$,
where $U_\varphi \subset \C$ is a certain open set described below.
It seems then natural to define the regularized Petersson inner product as the value of $I_\varphi(f, g; w, s)$ at $w=s=0$.
However, we want the inner product to be hermitian, so that $\langle f, f \rangle$ is real. The problem now is that
$I(f,f; 0, 0)$ is not real in general  because of the presence of the generalized exponential integrals in \eqref{eq:Ifgwcont} below.
To overcome this, we take the real part of this expression as a definition, which turns out to be natural.
For instance, it is independent of the choice of the analytic continuation.
Another difficulty is that $I_\varphi(f,g;w,s)$ is not analytic in $(0,0)$ if $f\cdot\ov{g}$ has a constant term.
In this case we show that there is a natural way to define $I_\varphi(f,g;0,s)$ and this function can be continued to a meromorphic function in
$s \in \C$. We then denote by $\CT\limits_{s=0}(I_\varphi(f, g; 0, s))$ the constant term in the Laurent expansion
of $I_\varphi(f, g; 0, s)$ at $s=0$.
\begin{definition}
\label{def:regprod}
For $\varphi \in (\pi/2, 3\pi/2) \setminus \{\pi\}$,
define
\begin{align*}
  \langle f, g \rangle_\varphi := \CT_{s=0}(I_\varphi(f,g;0,s)) - i\sum_{n > 0} c_f(-n) \cdot
\ov{c_g(-n)}\Im(E_{2-k,\varphi}(-4\pi n)).
\end{align*}
\end{definition}

The main result of this section is the following.
\begin{theorem} \label{thm:fgtindep}
  For each $f, g \in M^!_{k, \rho}$
  the value $\langle f, g \rangle := \langle f, g \rangle_\varphi$ is independent of the choice of $\varphi$.
  It satisfies $\langle f, g \rangle = \ov{\langle g, f \rangle}$ and, in particular, $\langle f, f \rangle \in \R$.
\end{theorem}

To prove Theorem \ref{thm:fgtindep}, $\varphi \in \R$, we define
\[
  R_{f,g}(\varphi):=\left\{ 4\pi n + e^{i\varphi}x\, \Big| \, x \in \R^+_{0}, n>0 \, \, \text{such that} \, \, c_f(-n)\cdot \ov{c_g(-n)} \neq 0 \right\}.
\]
\begin{proposition}\label{prop:Icont}
For $f,g \in M^!_{k,\rho}$
and
$\sigma \gg 0$,
the integral $I(f,g; w, s)$
defines a holomorphic function in a half-plane
$\zeta \gg 0$.
It can be analytically continued to a holomorphic function in any domain of the form $\{w \in \C\, \mid\, \zeta> -\varepsilon \} \setminus R_{f,g}(\varphi)$,
where $\varepsilon > 0$ depends on $f,g$ and  $\varphi \in (\pi/2,3\pi/2) \setminus\{\pi\}$.
This continuation is given by
\begin{multline}
  \label{eq:Ifgwcont}
 I_\varphi(f,g; w,s)   = \lim_{t \to \infty}  \left(\int_{\calF_t}f(\tau) \cdot \ov{g(\tau)}v^{k-s}e^{-wv}\, \frac{dudv}{v^2} \right. \\
\qquad \quad\left. - \sum_{n \geq  0}c_f(-n) \cdot \ov{c_g(-n)}\int_{1}^tv^{k-2-s}e^{(4\pi n-w)v}\, dv \right)
               +\sum_{n \geq 0} c_f(-n) \cdot \ov{c_g(-n)}E_{2-k+s, \varphi}(w-4\pi n).
\end{multline}
  \end{proposition}

  \begin{proof}
The integral \eqref{eq:Ifg} converges absolutely for
$\sigma, \zeta \gg 0$,
and thus defines a holomorphic function.
In the region of convergence we have
\begin{equation} \label{split}
I(f,g; w, s)=
\lim_{t \to \infty}
\left(\int_{\calF_1}f(\tau)  \cdot \ov{g(\tau)}v^{k-s}e^{-wv}\, \frac{dudv}{v^2}+
\int_{v=1}^t \int_{u=0}^1 f(\tau)  \cdot \ov{g(\tau)}v^{k-s}e^{-wv}\, \frac{dudv}{v^2} \right ).
\end{equation}
We insert the Fourier expansions of $f$ and $g$ and carry out the integration on $u$ to obtain
\begin{multline*}\int_{u=0}^1\int_{v=1}^t f(\tau) \cdot \ov{g(\tau)}v^{k-s}e^{-wv}\, \frac{dudv}{v^2} \\
=\sum_{n  > 0}c_f(n) \cdot \ov{c_g(n)} \int_{1}^t e^{-(w+4\pi n)v}v^{k-2-s}dv
+ \sum_{n \ge 0}c_f(-n) \cdot \ov{c_g(-n)} \int_{1}^t e^{(4\pi n-w)v}v^{k-2-s}dv.
\end{multline*}
Using Lemma \ref{lem:coeffgr} and the asymptotic behavior of the incomplete Gamma function,
we deduce that the limit of the second summand as $t\rightarrow\infty$ converges absolutely and uniformly in $w$
for $\zeta>-\varepsilon$ for some $\varepsilon > 0$.
Since, in addition, the integral over $\calF_1$ in \eqref{split} defines an entire function in $w$,
this implies that
\begin{multline} \label{split1}
\lim_{t \to \infty} \left(\int_{\calF_t}f(\tau) \cdot \ov{g(\tau)}v^{k-s}e^{-wv}\, \frac{dudv}{v^2}
                     - \sum_{n \ge 0}c_f(-n) \cdot \ov{c_g(-n)}\int_{1}^tv^{k-2-s}e^{(4\pi n-w)v}\, dv \right) \\
=
\int_{\calF_1}f(\tau) \cdot \ov{g(\tau)}v^{k-s}e^{-wv}\, \frac{dudv}{v^2}+
     \sum_{n > 0}c_f(n) \cdot \ov{c_g(n)} \int_{1}^\infty e^{-(4\pi n +w)v}v^{k-2-s}dv
\end{multline}
defines a holomorphic function for $\zeta>-\varepsilon$.

On the other hand, for $\sigma ,\zeta \gg 0$
$$\lim_{t \to \infty} \left (
\sum_{n \ge 0}c_f(-n) \cdot \ov{c_g(-n)} \int_{1}^t e^{(4\pi n -w)v}v^{k-2-s}dv \right )=
\sum_{n \ge 0}c_f(-n)  \cdot \ov{c_g(-n)} E_{2-k+s}(w-4\pi n).
$$
This is a finite sum and, for $\varphi \in (\pi/2,3\pi/2) \setminus\{\pi\}$, can be analytically continued
to $\mathbb C\backslash R_{f, g}(\varphi)$ to give the function
$\sum_{n \ge 0}c_f(-n)\cdot \ov{c_g(-n)} E_{2-k+s, \varphi}(w-4\pi n)$, implying the claim.
  \end{proof}

  \begin{remark}
    For $\sigma \gg 0$, the only term in \eqref{eq:Ifgwcont} which is not analytic at $w=0$ is
   \[
     c_f(0) \cdot \ov{c_g(0)}E_{2-k+s, \varphi}(w).
   \]
   It has, however, a well-defined value at $w=0$ as given in \eqref{eq:Es0}
   which we use to define $I_\varphi(f,g; 0, s)$.
   The function $s\mapsto I_\varphi(f,g; 0, s)$ then has a meromorphic continuation to $\C$
   since the first line in \eqref{eq:Ifgwcont} (for $w=0$) is entire and the function
   $s\mapsto E_{2-k+s}(-4\pi n)$ is entire for $n \neq 0$. For $n=0$ it has a simple pole at $s=k-1$.
  \end{remark}

  \begin{remark}
   Proposition \ref{prop:Icont} together with \eqref{eq:Es0} shows that, for $k \neq 1$, we have
    \begin{align}\label{eq:fgnoct}
      \langle f, g \rangle_\varphi &= \lim_{t \to \infty} \left(\int_{\calF_t}f(\tau) \cdot \ov{g(\tau)}v^{k}\, \frac{dudv}{v^2}
                                          - \sum_{n \geq 0}c_f(-n) \cdot \ov{c_g(-n)}\int_{1}^tv^{k-2}e^{4\pi nv}\, dv \right) \notag\\
                    &\phantom{=}+c_f(0) \cdot \ov{c_g(0)}\frac{1}{1-k} + \sum_{n > 0} c_f(-n) \cdot \ov{c_g(-n)}\Re(E_{2-k, \varphi}(-4\pi n)).
    \end{align}
For $k=1$ the term involving $c_f(0) \cdot \ov{c_g(0)}$ is not present since $\CT_{s=0}s^{-1} = 0$.
  \end{remark}

\begin{proof}[Proof of Theorem \ref{thm:fgtindep}]
The claim that  $\langle f, g \rangle$  is well-defined follows by Proposition \ref{prop:Icont}.
  The independence comes from the fact that the different branches of the
generalized exponential integral
  differ by a real multiple of $2\pi i$ at the point $-4\pi n$.
For $k\in \Z$, this follows from \eqref{eq:Enexp} and
\eqref{eq:E1exp} and for $k\in 1/2+\Z$ it is a consequence of \eqref{eq:Enexp}
and \eqref{eq:ReE12}.
\end{proof}

\section{Relation to Fourier coefficients of harmonic Maass forms
and applications}\label{Relationto}
\subsection{Relation to Fourier coefficients of harmonic Maass forms}
\label{RelationtoFou}

For $F \in \calH_{k,\rho}$ and $G \in \calH_{2-k,\overline\rho}$, we
define a bilinear pairing
\begin{equation}
\label{bilinear}
  \{F, G\} := \sum_{n \in \Q} c_F^+(n)  \cdot c_G^+(-n),
\end{equation}
which is always a finite sum. 
The following theorem generalizes Proposition 3.5 of \cite{BrF}.
\begin{theorem} \label{InnProdPairing}
  Let $f, g \in M^!_{k,\rho}$ and let $G \in \calH_{2-k,\overline\rho}$, such that  $\xi_{2-k}(G) = g$.
  Then we have
\[
  \langle f, g \rangle = \{f, G \}.
\]
\end{theorem}
\begin{remark}
 Existence of a $G \in \calH_{2-k,\overline\rho}$ with $\xi_{2-k}(G) = g$ follows from the exactness of \eqref{ex-sequ}.
\end{remark}
\begin{proof}
Using Definition \ref{def:regprod} and \eqref{eq:Ifgwcont}, we can write
  \begin{multline}
 \langle f, g \rangle
 = \lim_{t \to \infty}  \left(\int_{\calF_t}f(\tau) \cdot \ov{g(\tau)}v^{k}\, \frac{dudv}{v^2}
- \sum_{n \geq  0}c_f(-n) \cdot \ov{c_g(-n)}\int_{1}^tv^{k-2}e^{4\pi nv}\, dv \right) \\
               +\CT_{s=0}\left(c_f(0) \cdot \ov{c_g(0)}\Re(E_{2-k+s, \varphi}(0))
+ \sum_{n > 0} c_f(-n) \cdot \ov{c_g(-n)}\Re(E_{2-k+s, \varphi}(-4\pi n))\right).
\end{multline}
We apply Stokes' theorem, as in the proof of Proposition 3.5 of \cite{BrF},
 to the integral over $\calF_t$, which gives (for $k \neq 1$)
$$
  \int_{\calF_t}f(\tau) \cdot \ov{g(\tau)}v^{k}\, \frac{dudv}{v^2} =
   \{f, G \}  + c_f(0) \cdot c_G^-(0)t^{k-1} + \sum_{n \neq 0}c_f(-n)\cdot c_G^-(n)W_{2-k}(2\pi n t).
$$
Since $\lim\limits_{t \to \infty} W_{2-k}(2\pi n t) = 0$ for $n<0$, we obtain
\begin{align*}
  \langle f, g \rangle &= \{f, G \}  + \lim_{t \to \infty}c_f(0) \cdot \left(c_G^-(0)t^{k-1} -  \ov{c_g(0)}\int_{1}^tv^{k-2}\, dv \right)\\
              &\phantom{=}+ \sum_{n>0}c_f(-n)\cdot \lim_{t \to \infty} \left(c_G^-(n)W_{2-k}(2\pi n t) - \ov{c_g(-n)}\int_{1}^tv^{k-2}e^{4\pi nv}\, dv \right)\\
              &\phantom{=\ } +c_f(0) \cdot \ov{c_g(0)} \CT_{s=0}(E_{2-k+s}(0)) + \sum_{n>0}c_f(-n) \cdot \ov{c_g(-n)}\Re(E_{2-k}(-4\pi n)),
\end{align*}
where we dropped $\varphi$ since $\Re(E_{2-k}(-4\pi n))$ is independent of it.
Next, we use that for any $t$ and $n \ne 0$, we have
\begin{align*}
  -\int_{1}^tv^{k-2}e^{4\pi nv} dv = t^{k-1}E_{2-k, \varphi}(-4\pi nt) - E_{2-k,\varphi}(-4\pi n)
                                     = (-4\pi n)^{1-k}(W_{2-k}(2\pi n t) - W_{2-k}(2\pi n)).
\end{align*}
Then we substitute
$E_{2-k+s}(0) = \frac{1}{1-k+s}$ (see \eqref{eq:Es0}),
which is valid for $\sigma > k-1$.
Finally, we use \eqref{usefuleq} again, implying the statement of the theorem in this case.
For $k=1$, the proof is completely analogous using that $\ov{c_g(0)}\int_{1}^t \frac{dv}{v} = c_G^{-}(0)\log(t)$ and $\CT_{s=0} E_{1+s}(0) = 0$.
\end{proof}

The following corollary generalizes so-called \textit{Zagier duality}
between weakly holomorphic modular forms of weight $k$ and $2-k$.
\begin{corollary}
  We have for any $f \in M^!_{k,\rho}$ and $g \in M^{!}_{2-k,\ov{\rho}}$ that $\{f, g\} = 0.$
\end{corollary}

We next obtain information about the extent of non-degeneracy of our inner product.
Before we accomplish this in Corollary \ref{deg}, we state two lemmas. The first one is a
modularity criterion for formal power series which is well-known (see, for instance, Theorem 3.1 of \cite{BoGKZ} or Proposition 4.3 of \cite{brli-genjac}).
\begin{lemma}
  \label{lem:mod}
  Let $f(q) = \sum_{n \in \Q^+_0} c_f(n) q^n \in V\llbracket q^{N_\rho} \rrbracket$, such that (formally) $f \mid_k T = f \mid_k Z = f$.
  Then $f$ is the $q$-expansion of an element of $M_{k,\rho}$ if and only if  $\{ f, g \} = 0$ for all $g \in M_{2-k,\ov{\rho}}^!$.
\end{lemma}

\begin{remarks}\ 
  \begin{enumerate}[leftmargin=*]
  \item In Lemma \ref{lem:mod}, $\{f,g\}$ is defined formally as in \eqref{bilinear}.
  \item Using the exactness of \eqref{ex-sequ} and a variant of Theorem 3.1 of \cite{BoGKZ}, the existence of harmonic Maass forms
with prescribed principal parts can be shown as in Proposition 3.11 of \cite{BrF}.
  \end{enumerate}
\end{remarks}

\begin{lemma}\label{lem:exfp}
	For every Fourier polynomial $P(q) = \sum_{n<0}c_P(n) q^n$ satisfying $P \mid_k T =  P$
        and $P \mid_k Z = P$, 
       there exists a harmonic Maass form $F \in \calH_{k,\rho}$ with $\xi_k(F) \in S_{2-k,\ov{\rho}}$, such that $P = P_F$.
\end{lemma}

Using these results, we obtain a precise description of the space on which the pairing is degenerate.
\begin{corollary}\label{deg}
  Let $f \in M_{k,\rho}^!$.
  We have that $\langle f, g \rangle = 0$ for all $g \in M^!_{k,\rho}$ if and only if there exists an $F \in \calH_{2-k,\ov{\rho}}$ with $\xi_{2-k}(F)=f$ and $F^+ = 0$.
  (Equivalently, $F^+ \in M_{2-k,\ov\rho}^!$ for any $F$ with $\xi_{2-k}(F)=f$.)
\end{corollary}
\begin{proof}
  It is clear that if such an $F$ exists, then $\langle f, g \rangle = 0$ by Theorem \ref{InnProdPairing}.

  Now assume that  $\langle f, g \rangle = 0$ for all $g \in M^!_{k,\rho}$.
  This implies, in particular, that $\langle f, g \rangle = 0$ for every $g \in M_{k,\rho}$.
  Thus, $\{F,g \} = 0$ for every $F \in \calH_{2-k,\overline\rho}$ with $\xi_{2-k}(F) = f$
  and every $g \in M_{k,\rho}$ by Theorem \ref{InnProdPairing}. By Lemma \ref{lem:exfp}, we can assume that $a_F(n) = 0$ for every $n < 0$ by subtracting, if necessary, a harmonic Maass form $\widetilde{F}$ with the same principal part as $F$ and with $\xi_{2-k}(\widetilde{F}) \in S_{k,\rho}$.
   By Corollary 3.9 of \cite{BrF}, we must have that $\widetilde{F}$ is weakly holomorphic so that in fact $\xi_{2-k}(\widetilde{F}) = 0$. However, $\{F^+, g\} = \{F, g\} = 0$ for every $g \in M^!_{k,\rho}$ implies that $F^+$ is the $q$-expansion of a holomorphic modular form
  in $M_{2-k,\overline\rho}$ by Lemma \ref{lem:mod}. Thus, $F-F^+ \in \calH_{k,\ov{\rho}}$ with $\xi_{2-k}(F-F^+) = \xi_{2-k}(F) = f$ has a vanishing holomorphic part.
\end{proof}

Now consider the subspace $\calT := \calT_{2-k,\ov\rho} \subset \calH_{2-k,\ov\rho}$ spanned by all $F \in  \calH_{2-k,\ov\rho}$ with $F^+ \in M_{2-k,\ov\rho}^!$.
The sesquilinear form $\la \cdot, \cdot \ra$ induces a non-degenerate hermitian sesquilinear form on the quotient space $M_{2-k,\rho}^! / \xi_{2-k}(\calT)$.

\begin{remark}
  It can happen that $\xi_{2-k}(\calT) = M_{k,\rho}^!$; for instance this is the case if $\rho$ is the trivial representation and $k=2$.
\end{remark}

\subsection{Application: Weight 0}
\label{sec:weight-0}
A consequence of Theorem \ref{InnProdPairing} is a new proof of (\ref{fmfn}) and, further, Theorem \ref{thm: fmfn}.

In Proposition 5 of \cite{dit-weight2}, the Fourier expansions of a basis of $\calH_2$ was determined.
For the convenience of the reader, we state the result we need, using our notation.
For $m\in-\N$ the basis elements $F_m$ of $\calH_2$ are
weakly holomorphic, whereas for $m\in\N$, $F_m$ has an expansion of the form
\begin{equation} \label{h_m}
F_m(\tau)=:\mathcal{M}_m(v)e^{2 \pi i m u}-\frac{6}{\pi
v}\sigma_1(m)-\sum_{n<0}|n| c_m(|n|) \mathcal{W}_n(v)e^{2 \pi i n u}-\sum_{n>0}\mathcal{L}_{m,n}q^{n}.
\end{equation}
Proposition \ref{prop:MnE2} and \eqref{calW} immediately yield
\begin{corollary} For $m \in \N$, we have

\begin{equation*}
F_m(\tau)=-
\sum_{n>0} \mathcal L_{m, n} q^n-
\frac{6}{\pi v}
\sigma_1(m)+m W_2(2 \pi m v) q^m+\sum_{n<0} n c_m(-n)W_2(2 \pi n v) q^n.
\end{equation*}
\end{corollary}

For $m \in \N$ the basis elements
$F_m \in \calH_{2}$ and the basis elements $f_m \in M_0^!$, defined in \eqref{fm}
are linked by (see Proposition 5 of \cite{dit-weight2})
$$\xi_2(F_m)=\frac{1}{4\pi} f_m.$$

By Proposition 4 of \cite{dit-weight2}, we have
\begin{equation}\label{eq:Lmn}
  \calL_{m, n} = 2\pi \sqrt{mn} \sum_{c\geq 1} \frac{K(m,n;c)}{c} F\left(\frac{4\pi \sqrt{mn}}{c}\right).
\end{equation}

Theorem \ref{InnProdPairing} together with \eqref{eq:Lmn} then implies Theorem \ref{thm: fmfn}.

\subsection{Application: Weight $3/2$}
\label{sec:weight-3/2}
Following \cite{BIF} and \cite{DIT}, we define for every discriminant $d\in\N$ and
$f \in M_0^!$ the {\it trace}
\begin{equation}
  \label{eq:trd}
  \tr_d(f) := \frac{1}{2\pi}\sum_{Q \in \Gamma\bs\calQ_{d}} \int_{\Gamma_Q \bs C_Q}^\reg f(\tau)\frac{d\tau}{Q(\tau,1)},
\end{equation}
where $\calQ_d$ denotes the set of integral binary quadratic forms of discriminant $d$,
$\Gamma_Q \subset \Gamma=\SL_2(\Z)$ is the stabilizer of $Q$ in $\Gamma$, and the cycle integral is regularized as in (1.10) of \cite{BIF}. Note that a different regularization for these cycle integrals
has been studied in \cite{Andersen}.

Let $\rho = \rho_{A}$ be the Weil representation of $\Mp_2(\Z)$ as in Example \ref{ex:G0NWrep} for $N=1$, 
and let $\bm{g}_d \in M_{3/2,\rho}^!$ be the unique weakly holomorphic modular form having a Fourier expansion of the form
\[
  \bm{g}_d(\tau) = q^{-\frac{d}{4}}\frake_{d} + \sum_{\mu \in \Z/2\Z}\sum_{\substack{n \in \N_0 \\ n\, \equiv\, -\mu^2 \pmod{4}}} a_d(n)q^{\frac{n}{4}} \frake_\mu.
\]
To compare our result with the one in \cite{DIT} (and prove the statement of Theorem \ref{thm:g1intro} in the introduction), let
$f, g \in M_k^!(4)$ be weight $k$ weakly holomorphic forms for $\Gamma_0(4)$ whose Fourier 
coefficients vanish unless $(-1)^{k-1/2}n  \equiv 0,1 \pmod{4}$.
Moreover, let $\bm{f}, \bm{g} \in M_{k,\rho}^!$ be the corresponding vector-valued modular forms under the isomorphism given in Example \ref{ex:G0NWrep}. The map $f \mapsto \bm{f}$ gives $c_{\bm{f}}(n/4) = c_{f}(n)$.
In \cite{DIT}, the regularized inner product of $f$ and $g$ is defined as
\[
  \lim_{t\to\infty} \int_{\calF_t(4)} f(\tau)\ov{g(\tau)}v^k\, \frac{dudv}{v^2}
\]
whenever the limit exists.
Suppose that $c_{\bm{f}}(-n)\cdot\ov{c_{\bm{g}}(-n)} = 0$ for all $n \geq 0$.
Then we have, for $t \geq 2$,
\begin{equation}
  \label{eq:lim32}
    \lim_{t\to\infty} \int_{\calF_t(4)} f(\tau)\ov{g(\tau)}v^k\, \frac{dudv}{v^2}
     = \frac{3}{2}   \lim_{t\to\infty} \int_{\calF_t} \bm{f}(\tau) \cdot \ov{\bm{g}(\tau)}v^k\, \frac{dudv}{v^2} = \frac{3}{2} \langle \bm{f}, \bm{g} \rangle,
\end{equation}
where $\calF_t(4)$ is a truncated fundamental domain for $\Gamma_0(4)$ as in \cite{DIT} and the limit exists.
That the limit exists and equals our regularized inner product follows from \eqref{eq:fgnoct}.

To show the first equality, let $G$ be a harmonic Maass form with $\xi_{2-k}(G)=g$ and similarly $\bm{G} \in \calH_{2-k,\ov{\rho}}$
the corresponding vector-valued form (which satisfies $\xi_{2-k}(\bm{G}) = \bm{g}$).
We have
\[
  \bm{G}(\tau) = \sum_{\mu \in \{0,1\}}\sum_{\substack{n \in \Z \\ (-1)^{\frac32-k}n\, \equiv\, \mu^2 \pmod{4}}} c_G\left(n,\frac{v}{4}\right)q^{\frac{n}{4}} \frake_\mu,
\quad
  \bm{g}(\tau) = \sum_{\mu \in \{0,1\}}\sum_{\substack{n \in \Z \\ (-1)^{k-\frac12}n\, \equiv\, \mu^2 \pmod{4}}} c_g(n)q^{\frac{n}{4}} \frake_\mu
\]
and similarly for $f$ and $\bm{f}$.
Inserting Fourier expansions of $\bm{f}$ and $\bm{G}$ shows that Lemma 2 in \cite{DIT} is equivalent to (for $t \geq 2$)
\[
  \int_{\calF_t(4)} f(\tau)\ov{g(\tau)}v^k\, \frac{dudv}{v^2} = \int_{0}^1\bm{f}(u+4it)\cdot\bm{G}(u+4it)\, du
                  + \frac{1}{2}\int_{0}^{1}\bm{f}(u+it)\cdot\bm{G}(u+it)\, du
\]
and in the limit we obtain \eqref{eq:lim32}.
To see this, note that
\[
  \int_{0}^{1}\bm{f}(u+it)\bm{G}(u+it)\, du = \sum_{n \in \Z} c_{\bm{f}}\left(\frac{n}{4}\right) \cdot c_{\bm{G}}\left(\frac{-n}{4},\frac{t}{4}\right)
    = \sum_{n \in \Z} c_{f}(n) c_{G}\left(-n, \frac{t}{4}\right)
\]
and it is easily seen that this agrees with the contribution from the integral over $f^{\mathrm{e}}G^{\mathrm{e}}+f^{\mathrm{o}}G^{\mathrm{o}}$ as in loc. cit.
Similarly,
\[
  \int_{0}^{1}\bm{f}(u+4it)\bm{G}(u+4it)\, du = \sum_{n \in \Z} c_{\bm{f}}\left(\frac{n}{4}\right) \cdot c_{\bm{G}}\left(\frac{-n}{4}, t \right)
      = \sum_{n \in \Z} c_{f}(n)c_{G}(-n, t)
\]
and this agrees with the contribution from the integral over $fg$ in \cite{DIT}.
Under our assumption that $c_{\bm{f}}(-n)\cdot\ov{c_{\bm{g}}(-n)} = 0$, we see that the limit exists
and is equal to
\[
\sum_{n \in \Z} c_{f}(n) c_{G}^+(-n)
\]
in both cases.
Consequently, the regularized inner product of $\bm{f}$ and $\bm{g}$ agrees 
with the regularized inner product of $f$ and $g$ defined in \cite{DIT}
(and the usual Petersson inner product of $f$ and $g$ whenever it converges) up to a factor of $3/2$.
Therefore, we obtain a regularized inner product for scalar-valued weakly holomorphic modular forms
for $\Gamma_0(4)$ satisfying the plus-space condition that extends
the regularization in \cite{DIT} by defining $\langle f,g \rangle := \tfrac32 \langle \bm{f}, \bm{g} \rangle$.

\begin{proof}[Proof of Theorem \ref{thm:g1intro}]
   By (1.2), (1.6), and (1.8) of \cite{BIF}, used in conjunction with the isomorphism in Example \ref{ex:G0NWrep}, the function
\[
  \bm{G}_1(\tau) = \sum_{d>0} \tr_d(f_1)q^{\frac{d}{4}}\frake_d + \frac{i}{\sqrt{\pi}}W_{\frac12}\left(\frac{\pi v}{2}\right)q^{\frac{1}{4}}\frake_{1}
  + \sum_{d<0}\tr_d(f_1)W_{\frac12}\left(\frac{\pi d v}{2}\right)(\pi \abs{d})^{-\frac12}q^{\frac{d}{4}} \frake_{d} -4\sqrt{v}\frake_0
\]
is a harmonic Maass form of weight $1/2$ and representation $\ov{\rho}$ satisfying $\xi_{1/2}(\bm{G}_1) = \bm{g}_1$.
(This is the vector-valued analog of the identity stated in the introduction of \cite{BIF} and we used Lemma \ref{lem:beta12}
to rewrite the Fourier expansion using our conventions.)
Thus, by Theorem \ref{InnProdPairing}, we have $ \la \bm{g}_1, \bm{g}_1 \ra = \tr_1(f_{1})$.
Finally, the explicit formula follows from Theorem 1.1 of \cite{BIF}.
\end{proof}

Note that the expression in Theorem \ref{thm:g1intro} can be interpreted as the value at $s=0$ of an L-function for the modular function $f_1$.
We refer the reader to \cite{BIF} and Section \ref{sec:L}.

\section{A cohomological interpretation of the error of modularity}\label{coh}
\subsection{The error of modularity}\label{errmod}
We next turn to a cohomological interpretation of the error of modularity
of holomorphic parts of harmonic Maass forms.
This is motivated by a question posed in \cite{dit-weight2},
namely the geometric meaning of the inner product in \cite{dit-weight2}.

\emph{In this section, let $k \in -\frac{1}{2}\N_0$ and let $\rho$ be a one-dimensional representation of $\Mp_2(\Z)$.}
If $\rho$ satisfies $\rho(\mathrm{I}, -1) = (-1)^{2k}$, then, for $M = \kabcd \in \SL_{2}(\Z)$,
\[
 \nu(M) := \rho(M,\varphi) \varphi(\tau)^{2k} (c\tau+d)^{-k}
\]
is independent of the choice of $\varphi$ and defines a unitary multiplier system on $\SL_2(\Z)$
for weight $k$, as given in Section 1.2 of \cite{BCD}.
Therefore, the slash operator $\mid_{k,\rho}$ is equivalent to
the slash operator defined for $f\colon \mathbb{H}\to \mathbb{C}$ and $M = \kabcd \in \SL_2(\Z)$ by
\[
  (f\mid_{k,\nu}M) (\tau) := \nu(M)^{-1}(c\tau +d)^{-k} f(M\tau),
\]
i.e., if $(M,\varphi) \in \Mp_2(\Z)$, then $f\mid_{k,\rho} (M,\varphi) = f\mid_{k,\nu}M$, independent of $\varphi$.
By a slight abuse of notation, we thus use $\mid_{k,\rho}$ to denote the
action of $\SL_2(\Z)$ on functions. Note that $M_{k,\rho}^! =\{0\}$ unless $\rho(\mathrm{I}, -1) = (-1)^{2k}$
since $f(\tau) = (-1)^{2k}\rho(\mathrm{I},-1)f(\tau)$ for every $f \in M^!_{k,\rho}$ and so it is natural
to make this assumption throughout.

Let $f \in M^!_{k, \rho}$ with Fourier expansion
\[
  f(\tau) = \sum_{\substack{n\in \Q \\ n \gg -\infty}}c_f(n)q^n.
\]
By \eqref{ex-sequ}, there exists an $F\in\calH_{2-k, \overline \rho}$ such that $\xi_{2-k}(F)=f$.
We first express $F^-$ as a non-holomorphic Eichler integral. For this, we let
$f^c(\tau):=\overline{f(-\overline{\tau})}$ for $f: \H \to \C$.
Using the same notation as in Section \ref{sec:sf},  we set
\begin{align*}
  G_f(\tau) &:=-\frac{1}{(2i)^{k-1}}\int_{-\overline{\tau}}^{i \infty}\frac{f_o^c(z)}{(z+\tau)^{2-k}}dz
-\frac{1}{\left(-4\pi\right)^{k-1}}
\sum_{\substack{n\in \Q^+}}
        \frac{\overline{c_f(-n)}}{n^{k-1}} W_{2-k}(2 \pi n v) q^n ,\qquad{\text{where }}  \\
f_o(\tau) &:= f(\tau)-\sum_{\substack{n\in \Q^+}} c_f(-n)
q^{-n} .
\end{align*}

\begin{lemma}\label{nonholE}
If $f \in M_{k, \rho}^!$ $(k \in -\frac{1}{2}\N_0)$ and $F \in \calH_{2-k, \overline \rho}$ such that
$\xi_{2-k}(F)=f$, then we have
\begin{equation}\label{nonholin}
F^-(\tau)=G_f(\tau).
\end{equation}
\end{lemma}
\begin{proof}
With 8.8.13 of \cite{NIST}, we deduce that
\begin{equation} \label{xigm}
\xi_{2-k}(G_f) = f,
\end{equation}
in particular $G_f$ is harmonic. To complete the proof, it thus suffices to
show that $G_f$ has an expansion of the same type as
that of \eqref{deff-}.
Indeed, we first see that
\begin{equation*} \label{expansion}
 \int_{-\overline{\tau}}^{i
\infty}\frac{f_o^c(z)}{(z+\tau)^{2-k}}dz
= \frac{\overline{c_f(0)}}{(k-1)(2iv)^{1-k}} + (-2\pi i)^{1-k}
\sum_{\substack{n\in \Q^+}}
\overline{c_f(n)} n^{1-k} W_{2-k}\left(-2 \pi n v\right)e^{-2\pi in
\tau} ,
\end{equation*}
as claimed.
The remaining term of $G_f$
has the shape \eqref{deff-}.
\end{proof}
The next lemma justifies calling $G_f$ an Eichler integral. For
$w \in \C$ with $\Re(w) \gg 0$, set
$$\mathcal F_{\tau}(w):=\int_{-\overline{\tau}}^{i\infty}
\frac{
e^{i w z} f^c\left(z\right)}{(\tau
+z)^{2-k}}dz.$$
\begin{lemma} \label{mathcalF}
The function $\mathcal F_{\tau}(w)$,
originally defined for $\Re(w) \gg 0$,
can be analytically extended to $\mathbb{C}\backslash(-\infty, 2\pi M]$,
where $M\in\Q^+$ is maximal such that  $c_f(-M) \neq 0$. It can further
be extended continuously from above to $(-\infty, 2\pi M)$. For $w \in \mathbb{C}\backslash(-\infty, 2\pi M]$
we have:
\begin{multline}\label{cont}\mathcal F_{\tau}(w)= \int_{-\overline{\tau}}^{i\infty}
\frac{e^{i w z}f_o^c\left(z\right)}{\left(\tau
+z\right)^{2-k}}dz+ e^{\frac{\pi i}{2}(k-1)}
\sum_{\substack{n\in \Q^+}}
\overline{c_f(-n)}(w-2\pi n)^{1-k}
W_{2-k} (-v(w-2\pi n))e^{-i\tau(w-2\pi n )} \\
-\frac{e^{\frac{3\pi i}{2}(k-1)} \pi i}{\Gamma(2-k)}
\sum_{\substack{n\in \Q^+}}
\overline{c_f(-n)}(w-2\pi n)^{1-k} e^{-i\tau(w-2\pi n )},
\end{multline}
and in particular
\begin{equation}\label{F(0)}
\mathcal{F}_\tau(0) =-(2i)^{k-1}G_f(\tau)
-\frac{e^{\frac{3 \pi i }{2}(k-1)} \pi i}{\Gamma(2-k)}
\sum_{\substack{n\in \Q^+}}
\frac{\overline{c_f(-n)}}{(-2 \pi n)^{k-1}}q^{n}.
\end{equation}
\end{lemma}
\begin{proof}
We first note that, for $\Re(w) \gg 0$,
\begin{equation}\label{decomp}
\mathcal F_{\tau}(w)=
\int_{-\overline{\tau}}^{i\infty} \frac{e^{iw
z}f_o^c\left(z\right)}{\left(\tau
+z\right)^{2-k}}dz+
\sum_{\substack{n\in \Q^+}}
\overline{c_f(-n)} \int_{-\overline{\tau}}^{i\infty}
\frac{e^{iwz-2 \pi i
n z}}{\left(\tau + z\right)^{2-k}}dz.
\end{equation}
We make the change of variables $z = -\tau + \frac{iz_1}{w-2\pi n}$, to obtain
\begin{equation}
\label{analco}
\int_{-\overline{\tau}}^{i\infty} \frac{e^{i z (w-2 \pi n)}}{\left(\tau +
z\right)^{2-k}}dz
=i^{k-1} e^{-i \tau (w-2\pi n)}(w-2\pi
n)^{1-k}\Gamma(k-1, 2v(w-2\pi n)).
\end{equation}
With the choice of branch of the incomplete Gamma function in Section
\ref{sec:sf},
this implies that the value at $w=0$ of the
continuous extension of $\mathcal F_\tau(w)$ to $\R^-$ equals
$$\int_{-\overline{\tau}}^{i \infty}\frac{f_o^c(z)}{(z+\tau)^{2-k}}dz
+        \frac{1}{(2\pi i)^{k-1}}
\sum_{\substack{n\in \Q^+}}
        \frac{\overline{c_f(-n)}}{n^{k-1}} \Gamma(k-1, -4 \pi n v) q^n.
$$
Using Lemma \ref{lem:relationsofnonhol}, we deduce \eqref{cont}.
\end{proof}
We are now ready to give a formula for the error of modularity of $F^+$. This leads
to a cohomological interpretation.
\begin{proposition}\label{unreg}
With the notation of Lemma \ref{nonholE}, we have, for $\tau \in \H$
 \begin{align} \label{unregeq}
F_S(\tau):&=F^+(\tau)|_{2-k, \overline \rho}\left(S-\mathrm{I}\right) \\ \nonumber
&=\frac{-1}{(2i)^{k-1}}
\int_{i}^{i \infty}
f_o^c(z) \left(\left(z+ \tau\right)^{k-2}
-\left(z-\frac{1}{\tau}\right)^{k-2}\tau^{k-2}\nu (S)
\right)dz  \\ \nonumber
-\frac{1}{(-4 \pi )^{k-1}}
\sum_{\substack{n\in \Q^+}}&
\frac{\overline{c_f(-n)}}{n^{k-1}}\Big
(e^{2\pi i n \tau}W_{2-k}\left(- \pi i
n(\tau+i)\right)-\tau^{k-2}e^{-\frac{2\pi i
n}{\tau}}\nu (S)W_{2-k}\left(-\pi i
n\left(i-\frac{1}{\tau}\right)\right)
\Big ).
\end{align}
\end{proposition}
\begin{proof}
 The functions in \eqref{nonholin} are real-analytic. We slash the 
left-hand side of \eqref{nonholin} by $S-\mathrm{I}$. Writing $v
= (\tau-\overline{\tau})\frac{1}{2 i}$ we can view the resulting function
as a
function in two independent variables, namely $\tau$ and $\overline{\tau}$.
The identity
\begin{equation}\label{modular}
  F^+|_{2-k, \overline \rho}(S-\mathrm{I}) = - F^-|_{2-k, \overline \rho}(S-\mathrm{I}).
\end{equation}
and \eqref{Wcomplex} imply that the
 left-hand side of the resulting equation is holomorphic and thus independent of
$\overline{\tau}$. Therefore the same holds for its right hand
side.
So, using that $Si=i$ and setting $\ov{\tau}=-i$, we obtain
 \begin{multline*}
 F_S (\tau)
=\frac{-1}{(2i)^{k-1}} \int_{i}^{ i\infty}
{f_o^c(z)}\left(z-\frac{1}{\tau}\right)^{k-2}\tau^{k-2}\nu (S)
dz \\ -\frac{1}{(-4 \pi )^{k-1}}
\sum_{\substack{n\in \Q^+}}
\frac{\overline{c_f(-n)}}{n^{k-1}}e^{-\frac{2\pi i n}{\tau}}
W_{2-k}\left (-\pi i n\left(i-\frac{1}{\tau} \right ) \right )
\nu (S) \tau^{k-2} \\
+\frac{1}{(2i)^{k-1}}
\int_{i}^{i \infty} {f_o^c(z)}\left(\tau+z\right)^{k-2}dz
+ \frac{1}{(-4 \pi )^{k-1}}
\sum_{\substack{n\in \Q^+}}
\frac{\overline{c_f(-n)}}{n^{k-1}} e^{2\pi i
n \tau}
W_{2-k}\left(-\pi i n\left(\tau +i\right)\right).
\end{multline*}
Using \eqref{modular}, this implies the result. To obtain the form of the
last sum appearing in the statement of the proposition, we employ \eqref{usefuleq}.

\end{proof}

\subsection{Cohomology} \label{TheCoh}

In this section we show that the error of modularity $F_S$ has a cohomological interpretation.
The space it belongs to was introduced in Definition 1.13 of \cite{BCD}, which we now recall.
Note that we reverse the roles of the upper- and lower-half-plane in comparison to \cite{BCD}.
The two formalisms are equivalent via the involution $\iota$ given by
$\iota (f)(\tau):=\overline{f\left(\overline{\tau}\right)}.$ For the remainder of this section we fix $k \in -\frac{1}{2}\N_0$.

We first define the spaces of excised semi-analytic vectors.
Let $\mathfrak{a} \in \mathbb{P}^1(\Q):=\Q \cup \{i\infty\}$ be a cusp of $\SL_2(\Z)$ (here cusps are not assumed to be necessarily $\SL_2(\Z)$-inequivalent).
Suppose that $M(i\infty)=\mathfrak{a}$ for some $M \in$ $\SL_2(\Z).$
For $a,\varepsilon \in \mathbb{R}^+$, set
\[
  V_{\mathfrak{a}}(a, \varepsilon):= \{M\tau\in \H^-; |u| \le a, v <-\varepsilon\},
\]
where $\H^-$ is the lower half-plane.
For $E$ a finite set of cusps of $\SL_2(\Z)$
a set $\Omega \subset \mathbb{P}^1(\mathbb{C})$ is called an {\it
$E$-excised neighborhood} of $\H\cup \mathbb{P}^1(\mathbb{R})$ if there
exists a neighborhood $U$ of $\H\cup \mathbb{P}^1(\mathbb{R})$ and
pairs of positive reals $(a_{\ca}, \varepsilon_{\ca})$ ($\ca \in E$) such
that $U \backslash \bigcup_{\mathfrak{a} \in E} V_{\mathfrak{a}}(a_{\ca},
\varepsilon_{\ca}) \subset \Omega.$ An example is shown in Figure \ref{fig-excnbh}.

\begin{figure}[ht]
\[\setlength\unitlength{.9cm}
\begin{picture}(10,5)(0,-4)
\put(0,0){\line(1,0){10}}
\put(2.9,.1){$\mathfrak{a}_1$}
\put(4.9,.1){$\mathfrak{a}_2$}
\put(4,-.4){$\Omega$}
\put(.5,.1){$\Omega$}
\put(9.3,-.4){$\Omega$}
\put(3,0){\circle*{.1}}
\put(5,0){\circle*{.1}}
\put(3,-.8){\vector(0,1){.5}}
\put(2.8,-1.2){$V_{\mathfrak{a}_1}(a_{\mathfrak{a}_1}, \varepsilon_{\mathfrak{a}_1})$}
\put(4,-3){$V_{i \infty}(a_{i \infty}, \varepsilon_{i \infty})$}
\put(4.2,-3.5){\vector(0,-1){.6}}
\put(6,-2){not in $\Omega$}
\thicklines
\put(2.5,0){\oval(1,1)[rb]}
\put(3.5,0){\oval(1,1)[lb]}
\put(4.5,0){\oval(1,1)[rb]}
\put(5.5,0){\oval(1,1)[lb]}
\put(3.5,-.5){\line(1,0){1}}
\put(1,-1){\oval(1,1)[lt]}
\put(1,-.5){\line(1,0){1.5}}
\put(.5,-1){\line(0,-1){3}}
\put(8.5,-1){\oval(1,1)[rt]}
\put(9,-1){\line(0,-1){3}}
\put(5.5,-.5){\line(1,0){3}}
\end{picture}
\]
\caption{An $\{i\infty,\mathfrak{a}_1,\mathfrak{a}_2\}$-excised
neighborhood.}\label{fig-excnbh}
\end{figure}
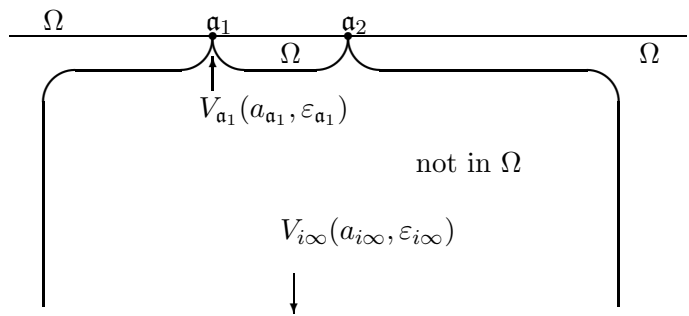
We also recall the map $\prj_{2-k}$, used in Section 1.5 of \cite{BCD} to move between the
projective model and the plane model. For any open subset $U \subset \mathbb{P}^1(\mathbb{C})$ not containing $-i$,
this map sends a function $f: U \to \C$
 to $\prj_{2-k}(f): U \to \C$ defined by
\[
  \prj_{2-k} (f)(\tau):=(i+ \tau)^{2-k}f( \tau).
\]
For an open subset $U \subset \mathbb{P}^1(\mathbb{C})$,
we let $\hol(U)$ be the space of holomorphic functions $U \to \C$ and
define the space of {\it excised semi-analytic vectors}
associated with a finite set $\{\mathfrak{a}_1, \mathfrak{a}_2,
\dots, \mathfrak{a}_n\}$ of cusps as
\[
D_{2-k, \overline{\rho}}[\mathfrak{a}_1, \mathfrak{a}_2, \dots, \mathfrak{a}_n]:= \prj_{2-k}^{-1}\indlim\hol(\Omega),
\]
where $\Omega$ runs over all $\{\mathfrak{a}_1, \mathfrak{a}_2, \dots,
\mathfrak{a}_n\}$-excised
neighborhoods.

\begin{example}[Special case $k \in -\N_0$]
An element of the space $D_{2-k, \overline{\rho}}[\mathfrak{a}_1,
\mathfrak{a}_2, \dots, \mathfrak{a}_n]$ is (represented by) a function
$f$ with the following properties:
\begin{enumerate}[leftmargin=*]
\item It is defined on some neighborhood $U$ of
$\H \cup \mathbb{P}^1(\mathbb{R})$ from which we have truncated some
``wedges" $V_{\mathfrak{a_1}}, \dots, V_{\mathfrak{a_n}}$.
\item It is holomorphic in $U$ unless the
neighborhood contains $i\infty$ in which case we allow
poles of order at most $2-k$ at $i\infty$.
\end{enumerate}
\end{example}

The main coefficient module we use is defined as the inductive limit
$$D_{2-k, \overline{\rho}}:=\indlim D_{2-k, \overline{\rho}}[\mathfrak{a}_1, \mathfrak{a}_2, \dots, \mathfrak{a}_n],$$
where $\{\mathfrak{a}_1, \mathfrak{a}_2, \dots, \mathfrak{a}_n\}$ ranges
over all finite sets of cusps of $\SL_2(\Z)$.

To illustrate the definition, assume again that $k\in -\N_0$.
Then an element of $D_{2-k, \overline{\rho}}$ is (represented by) a function $f$ with the following properties:
\begin{enumerate}[leftmargin=*]
\item It is defined on some neighborhood $U$ of $\H \cup \mathbb{P}^1(\mathbb{R})$ from which we have truncated
some ``wedge" $V_{\mathfrak{a_j}}$ for each cusp in some set of cusps $\{\mathfrak{a_1}, \dots, \mathfrak{a_n}\}$. \item It is holomorphic in $U$ unless the neighborhood contains $i\infty$.
In the latter case we allow poles of order at most $2-k$ at $i\infty$.
\end{enumerate}

In Proposition 1.14 of \cite{BCD}, it is proved that $D_{2-k, \overline{\rho}}$ is a
$\SL_2(\Z)$-module via $|_{2-k, \overline \rho}$.

The following Proposition is crucial for the cohomological interpretation of the error of modularity.
\begin{proposition}\label{g_m}
The function $F_S$
belongs to $D_{2-k, \overline{\rho}}$.
\end{proposition}
\begin{proof}
The holomorphicity of $\mathcal{F}_s$ on $\H$ follows by definition since $F^+$ is
holomorphic.

On the other hand, since $f_o(z)=c+O(e^{-\alpha y})$, as
$y \to \infty$, for some $c \in \C$ and $\alpha >0$, the integral on the right hand side
of \eqref{unregeq} defines a function which can be extended to a
holomorphic function everywhere except for the negative imaginary axis
(corresponding to the possible poles of the integrand).
To examine the terms involving the $W_{2-k}$-function we first note that, if $u \neq 0$, then
$- \pi i n(\tau+i)$ and
$- \pi i n(i-\frac{1}{\tau})$
do not lie on the real line. Thus the terms involving $W_{2-k}$
extend holomorphically to $\mathbb{C}\setminus i[0, -\infty)$. Using Proposition \ref{unreg}, we deduce that $F_S$
is a holomorphic function on some neighborhood of
$\H \cup \mathbb{P}^1(\mathbb{R})$ not containing $-i$ and with the
sectors $V_0$ and $V_{i \infty}$ excised.
Therefore, $F_S \in D_{2-k, \overline{\rho}}[0, i \infty]$
and, thus, by definition, it is in $D_{2-k, \overline{\rho}}$.
\end{proof}
To state our next theorem we recall the definition of parabolic cohomology groups.
\begin{definition}
  We let $Z^1_{\text{par}}({\SL}_2(\Z), D_{2-k, \overline{\rho}})$ be the
set of all maps $\phi: \SL_2(\Z) \to D_{2-k, \overline{\rho}}$, satisfying the conditions 
%{\bf KB: following looks real strange in formatting. I fixed a few other places. What happened with file?}
 $ \phi(M_1M_2)=\phi(M_1)|_{2-k, \overline \rho} M_2+\phi(M_2)$
for all $M_1, M_2 \in \SL_2(\Z)$ and
 $ \phi(T) \in D_{2-k, \overline{\rho}}|_{2-k, \overline \rho} (T-\mathrm{I}).$
We further set
\begin{align*}
B^1(\SL_2(\Z), D_{2-k, \overline{\rho}}) &:=
\{\phi : \SL_2(\Z) \to D_{2-k, \overline{\rho}};\ \phi(M)=a|_{2-k, \overline \rho}(M-\mathrm{I}) \,
\, \text{ with }a \in D_{2-k, \overline{\rho}} \}, \\
H^1_{\text{par}} \left(\SL_2(\Z), D_{2-k, \overline{\rho}} \right) &:=
Z^1_{\text{par}} \left(\SL_2(\Z), D_{2-k, \overline{\rho}} \right)\big/
B^1 \left(\SL_2(\Z), D_{2-k, \overline{\rho}} \right).
\end{align*}
\end{definition}
Note that since $F^+$ is invariant under $T$,
$F_S$ satisfies the period relations, i.e., it is annihilated in terms of the
action of $|_{2-k, \ov{\rho}}$ by $S+\mathrm{I}$ and $U^2+U+\mathrm{I}$ where $U:=TS$.
Thus the map $M \mapsto F^+|_{2-k, \overline \rho}(M - \mathrm{I})$ induces a parabolic $1$-cocycle, hence, by Proposition \ref{g_m}, a cohomology class.
\begin{theorem} \label{main}
The map $\sigma: \SL_2(\Z) \to D_{2-k, \overline{\rho}}$ defined via
$$\sigma(M) := (-2i)^{k+1} \pi F^+|_{2-k, \overline \rho}(M - \mathrm{I})$$
induces a cohomology class in $H^1_{\emph{par}}(\emph{\SL}_2(\Z), D_{2-k, \overline{\rho}})$.
\end{theorem}

We next show that the cohomology class of Theorem \ref{main} coincides with the
cohomology class attached to $f$ under one of the Eichler-Shimura
isomorphisms given in \cite{BCD}. In particular, this implies that the
cohomology class of Theorem \ref{main} is not trivial.
\begin{proposition}[Theorem E(i)(a) of \cite{BCD}]\label{ThE} Let
$k \in \R \setminus \N_{\ge 2}$,
and let $A_{k,\rho}$
denote the space of holomorphic
functions $f$ on $\H$ satisfying $f|_{k, \rho} M = f$
for all $M \in \SL_2(\Z)$.
Furthermore, for $\tau_0\in\H$ fixed, let $\ES_k$ be
the map assigning to $f \in A_{k,\rho}$ the map
\[
    M \mapsto \overline{\int_{M^{-1}
\tau_0}^{\tau_0}\frac{f(z)}{(z-\overline{\tau})^{2-k}}dz}
    =(-1)^{-k+1}\int_{-M^{-1}
\overline{\tau_0}}^{-\overline{\tau_0}}
\frac{f^c(z)}{(z+\tau)^{2-k}}dz,
       \qquad \qquad M \in  \SL_2(\Z).
\]
Then $\ES_k$ induces an isomorphism between $A_{k,\rho}$ and
$H^1_{\emph{par}}(\SL_2(\Z), D_{2-k, \overline{\rho}}).$
\end{proposition}
\noindent {\it Remark.} The image of $\ES_k(f)$ in
$H^1_{\text{par}}(\SL_2(\Z), D_{2-k, \overline{\rho}})$ is independent of $\tau_0.$

In order to prove that the cohomology class of $\sigma$ in fact coincides with the cohomology class of $\ES_k(f^c)$,
 we need the following explicit formula.
\begin{proposition} \label{relation}
Suppose that $2-k=m+b$ with $m \in
\N$ and $b\in \{1/2,1\}$. Then, for each $\tau \in \H$ with $u>0$, we have
\begin{multline}\label{relationeq}
-(2i)^{k-1} G_f(\tau)+\int_i^{-\overline{\tau}}\frac{{f^c(z)}}{(\tau+z)^{2-k}}dz
-\frac{e^{\frac{3 \pi i }{2}(k-1)} \pi i}{\Gamma(2-k)}
\sum_{\substack{n\in \Q^+}}\frac{\overline{c_f(-n)}}{(-2 \pi n)^{k-1}}q^{n}
\\
=\int_i^{i \infty}\frac{{f_o^c(z)}}{(z+\tau)^{2-k}}dz
+\sum_{\substack{n \in \Q^+}} \overline{c_f(-n)}
i^{k-1} e^{2 \pi i n \tau}(1-i\tau)^{k-1}
\Big ( e^{2 \pi n (1-i\tau)}
\frac{1}{\Gamma(2-k)} \\ \times \sum_{\ell=1}^{m-1}
\left(\left ( 2 \pi n (1-i \tau) \right )^{\ell}
\Gamma (1-k- \ell)+
\frac{\Gamma(c) \left ( 2 \pi n (1-i \tau) \right )^m}{\Gamma(2-k)}
E_b(2\pi n (i \tau-1))
\right ).
\end{multline}
 \end{proposition}
\begin{proof}
 Lemma \ref{mathcalF} yields that
$$-(2i)^{k-1} G_f(\tau)-\frac{e^{\frac{3 \pi i }{2}(k-1)} \pi
i}{\Gamma(2-k)}
\sum_{\substack{n\in \Q^+}}
\frac{\overline{c_f(-n)}}{(-2 \pi n)^{k-1}}q^{n}$$
equals the value at $w=0$ of the continuation of the function of $w$:
\begin{align}
  \int_{-\overline{\tau}}^{i\infty} & {f^c
\left(z\right)}e^{i w
z}\left(\tau +
z\right)^{k-2}dz \notag \\
&=
\int_{i}^{i\infty} \frac{
f_o^c\left(z\right)}{\left(\tau + z\right)^{2-k}}dz+
\sum_{\substack{n \in \Q^+}} \overline{c_f(-n)}
\int_{i}^{i\infty}
\frac{e^{i w z-2 \pi i nz}}{\left(\tau + z\right)^{2-k}}dz
-\int_i^{-\overline{\tau}} \frac{{f^c \left(z\right)}}{\left(\tau
+
z\right)^{2-k}}dz. \label{3pieces}
\end{align}
Equation 8.19.1 of \cite{NIST} gives for $\Re(w) \gg 0,$
 \begin{align}\label{exponent}
\int_{i}^{i\infty} \frac{e^{i w z-2 \pi i nz}}{\left(\tau +
z\right)^{2-k}}dz&=
i^{k-1} e^{-i \tau (w-2 \pi n)}(1-i \tau)^{k-1}
E_{2-k}((1-i \tau)(w-2 \pi n)) .
\end{align}
If $2-k=m+b$ with $b=1/2$ or $b=1$ and  $m \in \N$, then
\eqref{eq:Enexp} implies that (\ref{exponent}) equals
\begin{align*}
&i^{k-1} e^{-i \tau (w-2 \pi n)}(1-i \tau )^{k-1} \Big ( e^{(1-i \tau)(2 \pi n-w)}
\frac{1}{\Gamma(2-k)}\sum_{\ell=1}^{m-1}\left ( (1-i \tau)(2
\pi n-w)\right )^{
\ell}\Gamma (1-k-\ell) \\
&\hspace{60mm} +\frac{\Gamma(c)
\left ( (1-i \tau)(2 \pi n-w) \right )^m}{\Gamma(2-k)}
E_b((1-i \tau)(w-2 \pi n)) \Big ).
\end{align*}
 As a function of $\tau$ this is holomorphic for $w=0$ since
$
\text{Im}(2 \pi n(i\tau-1))=2 \pi n u >0.
$
Substituting this into \eqref{3pieces} and taking the limit as $w \rightarrow 0^+$, we obtain the double sum in \eqref{relationeq}. The remaining terms are obtained directly from
\eqref{3pieces}.
\end{proof}
Finally, we are ready to state the main theorem of this section.
\begin{theorem} \label{equalcoh} We assume the notation of
Proposition \ref{ThE}. Then the cohomology class of
$\ES_k(f^c)$ in $H^1_{\emph{par}}(\SL_2(\Z), D_{2-k, \overline{\rho}})$ equals the cohomology
class of the cocycle $\sigma$.
\end{theorem}

\begin{proof}
The statement of the theorem is equivalent to
\begin{equation}\label{2}
\ES_k(f^c)-\sigma \in
B^1(\emph{\emph{SL}}_2(\Z), D_{2-k, \overline{\rho}}).
\end{equation}
To prove \eqref{2}, we need to show that there exists $a \in D_{2-k, \overline{\rho}}$ such that, for each
$M \in \SL_2(\Z)$,
\begin{equation}\label{3}
\ES_k(f^c)(M)-\sigma(M)=a|_{2-k, \rho}(M-\mathrm{I}).
\end{equation}
The formulas for $\ES_k(f^c)(M)$ and $\sigma(M)$ imply that this is equivalent to
\begin{equation}\label{4}
\overline{\int_{M^{-1}\tau_0}^{\tau_0}\frac{{f^c(z)}}{(z-\overline{ \tau})^{2-k}}dz}
  +2^{k+1} \pi e^{\frac{\pi i}{2} (1-k)} F^+|_{2-k, \rho}(M-\mathrm{I})=a(\tau)|_{2-k,\rho}(M-\mathrm{I}).
\end{equation}
Using Lemma 5.1 of \cite{BCD}, we can rewrite \eqref{4} as
\begin{multline}\label{51}
\left ( \int_i^{-\overline{\tau}}
\frac{f^c(z)}{(\tau+z)^{2-k}}dz \right )
\Bigg|_{2-k, \overline \rho}(M-\mathrm{I})+2^{k+1}
\pi  e^{\frac{\pi i}{2} (1-k)} h^+(\tau)|_{2-k, \overline
\rho}(M-\mathrm{I}) =a(\tau)|_{2-k,
\overline \rho}(g-1)
\end{multline}
for all $\tau$ in some excised neighborhood of $\H\cup\mathbb{P}^1(\mathbb{R})$.

We next find an $a \in D_{2-k, \overline{\rho}}$ satisfying \eqref{51}.
 With the definitions of $m$ and $b$ as above, we let
\begin{multline*}\widetilde G_f (\tau):=
\int_{i}^{i \infty}\frac{f_o^c(z)}{(z+\tau)^{2-k}}dz+
i^{k-1}\sum_{\substack{n\in \Q^+}}
 \overline{c_f(-n)} e^{2 \pi i n \tau}(1-i\tau)^{k-1}
\Big ( e^{2 \pi n (1-i\tau)} \frac{1}{\Gamma(2-k)} \\
\times \sum_{\ell=1}^{m-1}\left ( 2 \pi n (1-i \tau)
\right )^{\ell}\Gamma (1-k- \ell)+
\frac{\Gamma(c)
\left ( 2 \pi n (1-i \tau) \right )^m}{\Gamma(2-k)}E_{b, 0}(2 \pi n (i
\tau-1)) \Big ) \\
+\frac{e^{\frac{3 \pi i }{2}(k-1)} \pi i}{\Gamma(2-k)}
\sum_{\substack{n\in \Q^+}}
\frac{\overline{c_f(-n)}}{(-2 \pi n)^{k-1}}q^{n}.
\end{multline*}
Recall,
from Section \ref{prelim}, that $E_{b, 0}$ denotes the branch of $E_b$ with the cut on the
non-negative real axis. We denote the logarithm branch used
in the definition of $E_{b, 0}$ by $\Log^-$.  Then
\begin{equation}\label{E_c, 0}
E_b(2 \pi n(i\tau-1))=E_{b, 0}(2 \pi n(i\tau-1)).
\end{equation}
Since for $\tau \in \H$ with $u>0$ we have $2 \pi n(i\tau-1) \in \H$, hence
$\Log^-(2 \pi n(i\tau-1))=\Log(2 \pi n(i\tau-1))$.
Together with \eqref{eq:E1exp} and \eqref{eq:erfc}, this implies \eqref{E_c, 0}.
Then, by Proposition \ref{relation}, we obtain
$$
-(2i)^{k-1} G_f(\tau)+\int_i^{-\overline{\tau}}
\frac{f^c(z)}{(\tau+z)^{2-k}}dz
=\widetilde G_f(\tau).$$
Therefore,
\begin{equation} \label{relationeq2}
-(2i)^{k-1}
G_f(\tau)|_{2-k, \overline \rho}(M-\mathrm{I})+
\left ( \int_i^{-\overline{\tau}}
\frac{{f^c(z)}}{(\tau+z)^{2-k}}dz \right
)\Bigg|_{2-k, \overline \rho}(M-\mathrm{I})=\widetilde G_f(\tau)|_{2-k,
\overline \rho}(M-\mathrm{I}).
\end{equation}

\indent
We next show that this identity extends to some excised neighborhood of
$\H \cup \mathbb{P}^1(\mathbb{R})$ and that $\widetilde G_f \in D_{2-k, \overline{\rho}}$.
More precisely, we prove that the two terms on the left-hand side of \eqref{relationeq2} and $\widetilde G_f$
can be analytically continued to an excised neighborhood of
$\H \cup \mathbb{P}^1(\mathbb{R})$.

First, in Theorem \ref{g_m}, it is shown that
$F^+|_{2-k, \overline \rho}(M-\mathrm{I})$ is in $D_{2-k, \overline{\rho}}$,
i.e., it is holomorphic on such a domain.
Next, according to Lemma 5.1. of \cite{BCD} and Proposition \ref{ThE},
 $$\left ( \int_i^{-\overline \tau} \frac{{f^c(z)}}{(\tau+z)^{2-k}}dz \right
)\Bigg|_{2-k, \overline \rho}(M-\mathrm{I})=(-1)^{-k+1} \ES_k(f^c)(M)(\tau)$$ is also in $D_{2-k, \overline{\rho}}$ and thus
holomorphic on a domain of the same kind.
Finally, $\widetilde G_f$ is also holomorphic on such a domain since:
\begin{itemize}[leftmargin=*, label=\textperiodcentered,nolistsep]
\item $\tau \mapsto E_{b, 0}(2 \pi n (i\tau-1))$ is
holomorphic everywhere except for $ (-i)[1, \infty)$ since the cut of
$E_{b, 0}$ is $[0, \infty)$,
\item $\int_{i}^{i \infty}{f^c(z)}(z+\tau)^{k-2}dz$ is
holomorphic everywhere except for a line joining $-i$ to $-i \infty$,
\item all other terms in the formula of $\widetilde G_f$
are holomorphic everywhere (except for $-i$).
\end{itemize}

\indent
Thus, we have proved that the three terms appearing in
\eqref{relationeq2} are analytic in a neighborhood of $\H \cup
\mathbb{P}^1(\mathbb{R})$ that does not contain $-i$ and from which we
have excised some sectors $V_0, V_{i \infty}, M^{-1}V_{ i \infty}$.
Hence, by uniqueness, \eqref{relationeq2} holds for all $\tau$ in such an
excised neighborhood of
$\H \cup \mathbb{P}^1(\mathbb{R})$.

The holomorphicity of $\widetilde G_f$
on an excised region just proved shows further
that
$$-\sigma(M)+\ES_k(f^c)(M)=a|_{2-k, \overline \rho}(M-\mathrm{I})$$
for some $a \in D_{2-k, \overline{\rho}}$
for all  $M \in $ $\SL_2(\mathbb Z)$, i.e., the cohomology classes of
$\sigma$ and $\ES_k(f^c)$ coincide.
\end{proof}
A direct consequence of Theorem \ref{equalcoh} is that the cocycle $\sigma$ is not cohomologically trivial.
\begin{corollary} The cohomology class of
$$\sigma :M \mapsto (-2i)^{k+1} \pi
F^+|_{2-k, \overline \rho}(M-\mathrm{I})$$
does not vanish in $H^1_{\emph{\text{par}}}(\SL_2(\Z), D_{2-k, \overline{\rho}})$.
\end{corollary}
\begin{proof} By Theorem \ref{equalcoh}, the cohomology class of $\sigma$
equals that of
$\ES_k(f^c)$. However, Theorem \ref{ThE} asserts that the map induced by
$\ES_k$ is an isomorphism. Since $f^c$ is non-trivial, we deduce the
non-vanishing of our cohomology class.
\end{proof}

\section{Connections to L-values}
\label{connections}
\label{sec:L}
For simplicity, we consider only integral weight $k$ in this section and also
assume that $\rho$ is trivial. In \cite{BFK, fricke-thesis} L-functions for weakly holomorphic modular forms
with vanishing constant Fourier coefficient were investigated. We briefly recall them in the
slightly more  general context that includes weight $0$ and weakly holomorphic modular forms
with non-vanishing constant term.

\begin{definition} Let $k \in \Z$ and let $g(\tau)=\sum_{\substack{m \ge m_0}} c_g(m) q^m
\in M_{k}^!$. Then, for $t_0>0$ fixed,
  we set
\begin{equation}\label{Ldef}
L_g^*(s):=\sum_{\substack{m \ge m_0 \\ m \ne 0}}\frac{c_g(m)\Gamma(s, 2
\pi m t_0)}{(2 \pi m)^s}+i^k
\sum_{\substack{m \ge m_0 \\ m \ne 0}}
\frac{c_g(m)\Gamma\left(k-s, \frac{2
\pi m}{t_0}\right)}{(2 \pi m)^{k-s}}-c_g(0) \left ( \frac{i^s t_0^{s-k}}{k-s} +\frac{t_0^s}{s}\right ).
\end{equation}

\end{definition}
\begin{remarks}\
\begin{enumerate}[leftmargin=*]
\item The growth estimates of Lemma
\ref{lem:coeffgr} imply that the two series are convergent.
\item One can show that this definition is independent of $t_0$.
\item For weight $0$, the value $L_g^*(0)$ agrees with the definition
made in (1.10) of \cite{BIF} for the ``central value'' of the L-function of a weakly holomorphic modular form of weight $0$
with vanishing constant term.
\end{enumerate}
\end{remarks}

The main result of this section concerns
\begin{multline}\label{partdefined}
\mathcal{G}_k(\tau):=\frac{-1}{(2i)^{k-1}}
\int_{i}^{i \infty} {\widetilde f( z)} \left(\left(z+\tau
\right)^{k-2}-
(z\tau-1)^{k-2} \right)dz
\\
+\frac{e^{-\pi i k}}{(4 \pi )^{k-1}}\sum_{j=1}^{m_0}
\frac{\overline{c_f(-j)}}{j^{k-1}}
e^{2\pi i j \tau}W_{2-k}(- \pi i j(\tau+i)),
\end{multline}
where $\widetilde{f}:=f^c_o-\overline{c_f(0)}$.
We note that the first term
has a well defined value as $\tau$ goes to $0$ from within $\H.$
Therefore, by Proposition \ref{unreg}, the function
$\mathcal{G}_k$ can be thought of as the part of
$F_S$
which is defined at $0$.
We show that, for $k\in -2\N$, the function
$\mathcal{G}_k$
is essentially the generating function of the L-values of $f^c$.
Specifically, although
this function has no Taylor expansion at $0$,
$\mathcal{G}_k^{(n)}(0)$
gives numbers explicitly involving values of
L-functions.

\begin{theorem}\label{Taylor} For $k \in -2\N_0$ and $n\in\N_0$,
\begin{multline}\label{LV}
\frac{\mathcal{G}_k^{(n)}(0)}{n!}=\frac{-(k-n-1)_{n} }{2^{k-1} i^{n+k}n! }
\Big (L^*_{f^c}(n+1)+\overline{c_f(0)} \left ( \frac{i^{n+1}}{k-n-1}
+\frac{1}{n+1}\right ) \\
-\sum_{j= 1}^{m_0} \overline{c_f(-j)} \frac{\Gamma(n+1, -2 \pi j)}{(-2 \pi
j)^{n+1}} \Big )
+\frac{2^{2-2k+n}\pi^{n-k+2}i^{n-1}}{ \Gamma(2-k)}\sum_{j=1}^{m_0}  \frac{\overline{c_f(-j)}}{j^{k-n-1}},
\end{multline}
where $(a)_n := \Gamma(a+n)/\Gamma(a)$ denotes the Pochhammer symbol.
\end{theorem}
\begin{remark}
The sum with the incomplete Gamma functions is explicit and
elementary since, for
$n \in \N,$ $\Gamma(n+1, w)$ has a closed formula.
\end{remark}
\begin{proof} We compute
\[
\frac{d^n}{d\tau^n}
\left(\left(z+\tau\right)^{k-2}- \left(z\tau-1\right)^{k-2}
\right) \\
=
(k-n-1)_{n}  \left(\left(z+\tau\right)^{k-2-n}
-z^n \left(z\tau-1\right)^{k-2-n}
\right).
\]
Letting $\tau \to 0$, and
substituting the Fourier expansion of
$\widetilde f$, (since $k$ is even) the contribution of the first term to $\mathcal{G}_k^{(n)}(0)$ is
\begin{multline}\label{Lf}
\frac{-(k-n-1)_{n} }{(2i)^{k-1}} \left(
\int_i^{i \infty}\widetilde f(z) z^{k-2-n}dz -
\int_{i}^{i \infty} \widetilde f(z) (-z)^n dz
\right)\\
=\frac{-(k-n-1)_{n}}{(2i)^{k-1}}
\left(
\sum_{\ell \ge 1}\overline{c_f(\ell)}\int_i^{i \infty}e^{2 \pi i\ell z}
z^{k-2-n}dz
-\sum_{\ell \ge
1}\overline{c_f(\ell)}\int_i^{i\infty}e^{2 \pi i \ell z}(-z)^n dz
\right ).
\end{multline}
Making the change of variable $z=it/(2 \pi \ell)$, this becomes
\begin{multline}\label{Lf1}
 \frac{(k-n-1)_{n}}{2^{k-1}i^{k+n+2}}
\left(
i^k\sum_{\ell \ge 1}\frac{\overline{c_f(\ell)}}{(2\pi
\ell)^{k-n-1}}\Gamma(k-n-1, 2 \pi \ell)
+\sum_{\ell \ge  1}\frac{\overline{c_f(\ell)}}{(2\pi
\ell)^{n+1}}\Gamma(n+1, 2 \pi \ell)
\right )\\
=
\frac{(k-n-1)_{n}}{2^{k-1} i^{k+n+2}}
\Bigg(L^*_{f^c}(n+1) +\overline{c_f(0)} \left (
\frac{i^{n+1}}{k-n-1}+\frac{1}{n+1}\right )
\\
-\sum_{j=1}^{m_0} \overline{c_f(-j)}
\frac{\Gamma(n+1, -2 \pi j)}{(-2 \pi j)^{n+1}}
-i^k
\sum_{j=1}^{m_0} \overline{c_f(-j)} \frac{\Gamma(k-n-1, -2 \pi j)}
{(-2 \pi j)^{k-n-1}}\Biggl ).
\end{multline}

To differentiate the second term in (\ref{partdefined}),  we first
see with 8.8.19 of \cite{NIST} that
\begin{equation*}\label{derG}
\begin{split}
\frac{d^n}{d \tau ^n}\left(e^{2\pi i j\tau}\Gamma\left(k-1,2 \pi i
j(\tau+i)\right) \right )
&=
 (2 \pi i j)^n (k-n-1)_n e^{2\pi ij \tau} \Gamma(k-n-1, 2 \pi i
j(\tau+i)).
\end{split}
 \end{equation*}
With this and \eqref{Wcomplex}, we get
\begin{multline}
\frac{d^n}{d \tau ^n}\left[ \frac{1}{(4 \pi )^{k-1}}\sum_{j=1}^{m_0}
\frac{\overline{c_f(-j)}}{j^{k-1}}
e^{2\pi i j \tau}W_{2-k}(- \pi i j(\tau+i)) \right]_{\tau=0}\\
=\frac{(k-n-1)_{n}}{2^{k-1}i^{n+2}}
\sum_{j= 1}^{m_0} \overline{c_f(-j)} \frac{\Gamma(k-n-1, -2 \pi j)}{(-2
\pi j)^{k-n-1}}
+\frac{4^{k-1}\pi^{2-k}}{i \Gamma(2-k)}\sum_{j=1}^{m_0} (2 \pi i j)^n \frac{\overline{c_f(-j)}}{j^{k-1}}.
\end{multline}
Upon addition to \eqref{Lf1}, the first term of the right-hand side cancels the last sum of \eqref{Lf1}.
The second term equals the last term of \eqref{LV}, completing the proof.
\end{proof}

\printbibliography

\end{document}